\documentclass[11pt,a4paper,oneside,reqno]{amsart}

\usepackage{latexsym}
\usepackage{amsmath}
\usepackage{amsthm}
\usepackage{amssymb}
\usepackage{amsfonts}
\usepackage{fancyhdr}
\usepackage{comment}
\usepackage{mathrsfs}                           % For different calligraphic fonts

\theoremstyle{definition}
\newtheorem{thm}{Theorem}[section]

\newtheorem{prop}[thm]{Proposition}
\newtheorem{cor}[thm]{Corollary}
\newtheorem{lemma}[thm]{Lemma}

\newtheorem*{remark}{Remark}

\numberwithin{equation}{section}

\newcommand{\p}{\partial}

\renewcommand{\l}{\left}
\renewcommand{\r}{\right}

\renewcommand{\det}[1]{~\mbox{det}\left(#1\right)}
\newcommand{\eps}{\varepsilon}

\newcommand{\mR}{\mathbb{R}}

\newcommand{\abs}[1]{\lvert #1 \rvert}
\newcommand{\norm}[1]{\lVert #1 \rVert}
\newcommand{\br}[1]{\langle #1 \rangle}

\newcommand{\mA}{\mathcal{A}}
\newcommand{\mS}{\mathscr{S}}
\newcommand{\mdiv}{\mathrm{div}}

\title[$p$-harmonic coordinates for H\"older metrics]{$p$-harmonic coordinates for H\"older metrics and applications}
\author{Vesa Julin}
\address{Department of Mathematics and Statistics, University of Jyv\"askyl\"a}
\email{vesa.julin@jyu.fi}
\author{Tony Liimatainen}
\address{Department of Mathematics and Statistics, University of Jyv\"askyl\"a}
\email{tony.t.liimatainen@jyu.fi}
\author{Mikko Salo}
\address{Department of Mathematics and Statistics, University of Jyv\"askyl\"a}
\email{mikko.j.salo@jyu.fi}

\begin{document}

\begin{abstract}
We show that on any Riemannian manifold with H\"older continuous metric tensor, there exists a $p$-harmonic coordinate system near any point. When $p=n$ this leads to a useful gauge condition for regularity results in conformal geometry. As applications, we show that any conformal mapping between manifolds having $C^{\alpha}$ metric tensors is $C^{1+\alpha}$ regular, and that a manifold with $W^{1,n} \cap C^{\alpha}$ metric tensor and with vanishing Weyl tensor is locally conformally flat if $n \geq 4$. The results extend the works \cite{LS_MRL, LS_Bach} from the case of $C^{1+\alpha}$ metrics to the H\"older continuous case. In an appendix, we also develop some regularity results for overdetermined elliptic systems in divergence form.
\end{abstract}

\maketitle

\vspace{-5pt}

\section{Introduction} \label{sec_intro}

Various regularity results in Riemannian geometry can be achieved via the use of harmonic coordinates. If $(M,g)$ is a Riemannian manifold, a local coordinate system $x = (x^1,\ldots,x^n)$ is called harmonic if 
\[
\Delta_g x^j = 0, \qquad 1 \leq j \leq n,
\]
where $\Delta_g$ is the Laplace-Beltrami operator. The notion of harmonic coordinates goes back to Einstein \cite{Einstein}, and these coordinates have the property that the Ricci and Riemann curvature tensors become elliptic operators when written in harmonic coordinates \cite{DeTurck_Kazdan}. The Ricci and Riemann tensors are invariant under diffeomorphisms, and fixing a harmonic coordinate system may be thought of as a gauge condition  which makes these operators elliptic. Harmonic coordinates are useful when studying regularity of isometries or flatness of low regularity metrics \cite{Taylor}. Also, the DeTurck trick for local wellposedness of the Ricci flow \cite{ChowKnopf} is related to harmonic coordinates, see \cite{GrahamLee}.

Harmonic coordinates are well suited to studying regularity results in Riemannian geometry, due to the fact that harmonic functions are preserved by isometries. In this paper we are interested in regularity results in conformal geometry. For this purpose it is useful to have a gauge condition which is preserved by conformal transformations. Such a condition was introduced in the recent work \cite{LS_MRL}, which 
established the existence of $p$-harmonic coordinates on any smooth manifold with $C^s$, $s > 1$, Riemannian metric. If $p=n$, the corresponding $n$-harmonic gauge condition is conformally invariant, and \cite{LS_MRL} applied $n$-harmonic coordinates to give a new proof of certain regularity results for conformal mappings. The subsequent article \cite{LS_Bach} showed that $n$-harmonic coordinates provide a gauge condition for ellipticity of conformal curvature tensors, and characterized local conformal flatness of manifolds of dimension $\geq 4$ having $C^s$, $s > 1$, metric tensor.

The aim of the present article is to extend the results of \cite{LS_MRL} to $C^s$, $s > 0$, metric tensors and the local conformal flatness characterization of \cite{LS_Bach} to $W^{1,n} \cap C^s$, $s > 0$,  metric tensors. Let us state the desired results in more detail.

In this paper $M$ will be a smooth ($=C^{\infty}$) manifold of dimension $n \geq 2$, and $g$ will be a Riemannian metric on $M$ with finite regularity. If $1 < p < \infty$, the $p$-Laplace operator on $(M,g)$ is the operator $u \mapsto -\delta(\abs{du}^{p-2} du)$ where $d$ is the exterior derivative, $\delta$ is the codifferential and $\abs{du}$ is the $g$-norm of the cotangent vector $du$. This operator is given in local coordinates by 
\begin{equation} \label{p_harm_eq}
-\delta(\abs{du}^{p-2} du) = \abs{g}^{-1/2} \partial_j (\abs{g}^{1/2} g^{jk} (g^{ab} \partial_a u \partial_b u)^{\frac{p-2}{2}} \partial_k u)
\end{equation}
where $(g_{jk})$ is the metric in local coordinates, $(g^{jk})$ is the inverse matrix of $(g_{jk})$, and $\abs{g} = \mathrm{det}(g_{jk})$. Throughout the paper we use the Einstein summation convention, where a repeated index in upper and lower position is summed from $1$ to $n$. If $p=2$, \eqref{p_harm_eq} is the usual Laplace-Beltrami operator, whereas for $p \neq 2$ the $p$-Laplace operator is a quasilinear degenerate elliptic operator. We refer to \cite{HKM, Lindqvist} for more details on $p$-Laplace and more general $\mA$-harmonic operators.

We will state our results in terms of $C^s$ spaces for $s \geq 0$. If $s$ is a nonnegative integer, we understand that $C^s(\mR^n)$ is the space of functions in $\mR^n$ whose derivatives up to order $s$ are bounded and continuous. If $s = k+\alpha$ where $k \geq 0$ is an integer and $0 < \alpha < 1$, then $C^s = C^{k+\alpha}$ is the usual H\"older space with norm 
\[
\norm{f}_{C^{k+\alpha}} = \sum_{\abs{\gamma} \leq k} \norm{\partial^{\gamma} f}_{L^{\infty}(\mR^n)} + \sum_{\abs{\gamma} = k} \sup_{x \neq y} \frac{\abs{\partial^{\gamma} f(x) - \partial^{\gamma} f(y)}}{\abs{x-y}^{\alpha}}.
\]
We will also consider the Zygmund spaces $C^s_*(\mR^n)$ for $s \in \mR$, which are defined via the norm 
\[
\norm{f}_{C^s_*} = \sup_{j \geq 0} 2^{js} \norm{\psi_j(D) f}_{L^{\infty}(\mR^n)}
\]
where $(\psi_j(\xi))_{j=0}^{\infty}$ is a Littlewood-Paley partition of unity. One has 
\[
C^s = C^s_*, \qquad \text{$s > 0$ not an integer},
\]
and $C^k \subset C^k_*$ if $k$ is a nonnegative integer. Most often we will use local versions of the above spaces, and in local coordinate statements will write $C^s$ instead of $C^s_{\mathrm{loc}}$ etc. See \cite{Taylor3} for more details on the above spaces.

The first main theorem states that whenever $(M,g)$ is a Riemannian manifold with $C^s$ metric tensor where $s > 0$, then near any point of $M$ there exist local coordinates all of whose coordinate functions are $p$-harmonic. This extends a result from \cite{LS_MRL} from the case $s > 1$ to $s > 0$. 

\begin{thm}[$p$-harmonic coordinates] \label{mainthm_pharmonic_coordinates}
Let $(M,g)$ be a Riemannian manifold whose metric is of class $C^s$, $s>0$, in a local coordinate chart about some point $x_0 \in M$. Let also $1 < p < \infty$. There exists a local coordinate chart near $x_0$ whose coordinate functions are $p$-harmonic and have $C^{s+1}_*$ regularity, and given any $\eps > 0$, one can arrange so that metric satisfies $\abs{g_{jk}(x_0) - \delta_{jk}} <  \eps$ for $j,k=1,\ldots,n$. Moreover, all $p$-harmonic coordinates near $x_0$ have $C^{s+1}_*$ regularity.
\end{thm}

The next result states that if $\phi$ is a $1$-quasiregular map between two manifolds with $C^s$ metric tensors where $s > 0$ (in particular $\phi$ could be a $C^1$ conformal map), then $\phi$ is in fact a $C^{s+1}$ local conformal diffeomorphism. This extends the results of \cite{Iwaniec_thesis, LS_MRL} from the case $s > 1$ to $s > 0$. The earlier result \cite{Ferrand} considered $C^{\infty}$ Riemannian metrics, and \cite{Taylor, CapognaLeDonne} give analogous results for Riemannian and subRiemannian isometries.

\begin{thm}[Regularity of conformal maps] \label{mainthm_conformal_regularity}
Let $(M,g)$ and $(N,h)$ be Riemannian manifolds, $n\geq 3$, where $g, h\in C^s$ for some $s>0$, $s \neq 1$. Let $\phi: M\rightarrow N$ be a non-constant mapping. Then the following are equivalent:
\begin{align}
 &\phi  \mbox{ is a Riemannian $1$-quasiregular mapping}, \label{1qc_anal} \\
 &\phi  \mbox{ is locally bi-Lipschitz and } \phi^*h=c\,g \mbox{ a.e.,}  \label{weak_form} \\ 
 &\phi  \mbox{ is a local $C^1$ diffeomorphism and } \phi^*h=c\,g, \label{1dif} \\
 &\phi  \mbox{ is a local $C^{s+1}$ diffeomorphism and } \phi^*h=c\,g. \label{rdif}
\end{align}
\end{thm}

The third result states that the vanishing of the Weyl tensor on a manifold with $W^{1,n} \cap C^{\alpha}$ metric tensor is equivalent with local conformal flatness. The Weyl tensor $W(g)$ of $(M,g)$ is the $4$-tensor 
\[
W_{abcd} = R_{abcd} + P_{ac} g_{bd} - P_{bc} g_{ad} + P_{bd} g_{ac} - P_{ad} g_{bc}
\]
where $R_{abcd} = g(R(\partial_a, \partial_b)\partial_c, \partial_d)$ is the Riemann curvature tensor, and $P_{ab}$ is the Schouten tensor given by 
\[
P_{ab} = \frac{1}{n-2} \left( R_{ab} - \frac{R}{2(n-1)} g_{ab} \right)
\]
where $R_{bc} = g^{ad} R_{abcd}$ is the Ricci tensor and $R = g^{rs} R_{rs}$ is the scalar curvature. If $n=3$ the Weyl tensor is always zero, and if $n \geq 4$ the Weyl tensor is conformally invariant in the sense that 
\[
W(cg) = c W(g)
\]
if $c$ is a smooth positive function. It is a classical result that for $C^3$ metrics, the vanishing of the Weyl tensor is equivalent with the manifold being locally conformally flat in dimensions $n \geq 4$ (see for instance \cite[Chapter 4]{Aubin}). The question of having a low regularity counterpart for this result has been raised in \cite[Section 2.7]{Iwaniec} and \cite[Section 1]{Martin}. Such a result was proved for $C^s$ metrics with $s > 1$ in \cite{LS_Bach}. Here we improve it to $W^{1,n} \cap C^s$ metrics with $s > 0$, where $W^{k,p}$ is the usual (local) Sobolev space of functions whose derivatives up to order $k$ are (locally) in $L^p$.

\begin{thm}[Local conformal flatness] \label{mainthm_weyl_flatness}
Let $M$ be a smooth $n$-dimensional manifold with $n \geq 4$, and let $g \in W^{1,n} \cap C^{\alpha}$ in some local coordinates near $x_0 \in M$ where $0 < \alpha < 1$. If $W(g) = 0$ near $x_0$, then $g_{ab} = c \delta_{ab}$ for some positive $W^{1,n} \cap C^{\alpha}$ function $c$ in some $n$-harmonic coordinates near $x_0$.
\end{thm}

In three dimensions, local conformal flatness of sufficiently regular metrics is characterized by the vanishing of the Cotton tensor. An analogue of Theorem \ref{mainthm_weyl_flatness} for $n=3$ was proved in \cite{LS_Bach} for $C^s$ metrics with $s > 2$. It is likely that the regularity assumption could be improved slightly, but we will not consider this here.

We also remark that there are alternative gauge conditions for ellipticity of conformal curvature tensors, based on solving the Yamabe problem locally to achieve constant scalar curvature and using harmonic coordinates. This procedure may require more regularity than the method of $n$-harmonic coordinates presented in this paper, see \cite{LS_Bach} for further discussion.

This paper is organized as follows. Section \ref{sec_intro} is the introduction. In Section \ref{sec_pharmonic_coordinates} we prove Theorem \ref{mainthm_pharmonic_coordinates}, and more generally we show the existence of $\mA$-harmonic coordinates for any $\mA$ satisfying certain structural assumptions including H\"older continuity in $x$. The main point in extending the result of \cite{LS_MRL} from $s > 1$ to $s >0$ is that instead of working with strong solutions and nondivergence form equations, we keep the equations in divergence form and apply Campanato and Schauder estimates  in order to show regularity of weak solutions. In Section \ref{sec_pharmonic_coordinates} we also give an additional $W^{2,q} \cap C^{1+\alpha}$ regularity result for the $p$-harmonic coordinates if $g$ is $W^{1,q} \cap C^{\alpha}$ where $0 < \alpha < 1$ and $q \geq 2$.

Section \ref{sec_regularity_conformal} considers the regularity of conformal mappings and establishes Theorem \ref{mainthm_conformal_regularity}. Given Theorem \ref{mainthm_pharmonic_coordinates}, the argument is analogous to \cite{LS_MRL} and is based on representing the conformal map locally in $n$-harmonic coordinates in the target manifold. We also supply some details related to integer values of $s$ that were missing in \cite{LS_MRL}.

In Section \ref{sec_weyl} we discuss the Weyl tensor for low regularity metrics and prove Theorem \ref{mainthm_weyl_flatness}. This uses the fact proved in \cite{LS_Bach} that the Weyl tensor becomes elliptic in $n$-harmonic coordinates after a conformal normalization. The vanishing of the Weyl tensor can be interpreted as a quasilinear overdetermined elliptic system involving a nonlinearity with quadratic growth, and to deal with this we will require regularity results for overdetermined elliptic systems.

There are many regularity results for square elliptic systems available in the literature, see for instance \cite{Giaquinta, GM}. In the case of sufficiently smooth coefficients, a linear overdetermined elliptic system $Pu = f$ can be reduced to the square system $P^* P u = P^* f$, and regularity results will follow from the corresponding results for square systems. Some results for overdetermined systems with nonsmooth coefficients (mostly in nondivergence form) are given in \cite{Taylor3, TaylorToolsForPDE, LS_Bach}. However, we need to deal with certain systems in divergence form which do not seem to be covered by these results. For this reason, in Appendix \ref{sec_appendix_regularity} we establish the basic Calder\'on-Zygmund and Schauder regularity results for linear overdetermined elliptic systems in divergence form. Complete proofs for these results are included, based on freezing coefficients and Fourier arguments.

\subsection*{Acknowledgements}

V.J.\ was supported by the Academy of Finland grant  268393. T.L.\ and M.S.\ were partly supported by the Academy of Finland (Centre of Excellence in Inverse Problems Research) and an ERC Starting Grant (grant agreement no 307023). M.S.\ was also supported by CNRS, and T.L.\ and M.S.\ would like to acknowledge support from the Institut Henri Poincar\'e Program on Inverse Problems in 2015.

\section{$p$-harmonic coordinates} \label{sec_pharmonic_coordinates}

In this section we will prove Theorem \ref{mainthm_pharmonic_coordinates}. Since this theorem is a local statement, it is sufficient to consider a corresponding result in $\mR^n$. 
Locally, $p$-harmonic functions in $(M,g)$ are solutions of the $\mA$-harmonic equation
\begin{equation} \label{A harmoninen}
\mdiv \,\mA(x,\nabla u) = 0
\end{equation}
in some open subset $\Omega \subset \mR^n$ where 
\begin{equation} \label{A_def}
\mA^j(x,\xi) = \abs{g(x)}^{1/2} g^{jk}(x) (g^{ab}(x) \xi_a \xi_b)^{\frac{p-2}{2}} \xi_k.
\end{equation}
The next result shows the existence of $\mA$-harmonic coordinates for any $\mA$ which satisfies the following  structural assumptions for some $M, \delta > 0$:
\begin{gather}
\mA(x,t\xi) = t^{p-1} \mA(x,\xi) \text{ for $t \geq 0$}, \label{mA_assumptions1} \\
\abs{\partial_{\xi_j} \mA^k(x,\xi)} \leq M \abs{\xi}^{p-2}, \label{mA_assumptionsoneandsome} \\
\abs{\mA(x,\xi) - \mA(y,\xi)} \leq M \abs{x-y}^\alpha \abs{\xi}^{p-1}, \label{mA_assumptionsoneandhalf}\\
(\mA(x,\xi)-\mA(x,\zeta)) \cdot (\xi-\zeta) \geq \delta (\abs{\xi} + \abs{\zeta})^{p-2} \abs{\xi-\zeta}^2. \label{mA_assumptions2}
\end{gather}

\begin{thm}[$\mathcal{A}$-harmonic coordinates]
\label{A-harm-coord}
Let $\Omega$ be an open set in $\mR^n$ and let $1 < p < \infty$ and  $0 < \alpha < 1$. Suppose 
$\mA: \Omega \times \mR^n \to \mR^n$ is a vector field which satisfies  \eqref{mA_assumptions1}--\eqref{mA_assumptions2}.
Given any point $x_0 \in \Omega$, there is an $C^{1+\alpha}$ diffeomorphism $U$ from some neighborhood of $x_0$ onto an open set in $\mR^n$ such that all coordinate functions of $U$ are $\mA$-harmonic. 
%If additionally\sn{Tarvitaan $C^r$ myos $\xi$ suhteen. Vrt. Ch. 4 Thm 6.4 of~\cite{Uraltseva}.}
%\begin{equation} \label{aharmcr}
% \mA \text{ is a $C^r$ map for every } \xi \in \mR^n - \{0\},
%\end{equation}
%then $U$ is a $C^{r+1}$ map and 
Moreover, any $C^1$ diffeomorphism whose coordinate functions are $\mA$-harmonic has this regularity. Given any invertible matrix $S$ and any $\eps > 0$, there is $U$ defined near $x_0$ such that %\sn{Tarvitaan yleiselle matriisille $S$ kuten oli jossain aiemmassa versioissa. Kts. Lause~\ref{RmA-harmonic}.}
$$
\norm{DU(x_0) - S} < \eps.
$$
%The coordinates can be chosen such that the Jacobian matrix at $x_0$ is any prescribed invertible matrix $S$ and such $x_0$ gets mapped to any prescribed point $y\in R^n$: 
%\begin{gather*}\label{euclid_at_origin}
% DU(x_0)=S, \\
% U(x_0)=y.
%\end{gather*}
\end{thm}

We remark that \eqref{mA_assumptions1}--\eqref{mA_assumptions2} imply the conditions (for $x \in \Omega$ and $\xi, h \in \mR^n$) 
\begin{gather*}
\mA(x,\xi) \cdot \xi \geq \delta \abs{\xi}^p, \\
\partial_{\xi_k} \mA^j(x,\xi) h^j h^k \geq \delta \abs{\xi}^{p-2} \abs{h}^2,
\end{gather*}
which appear in the $C^{1 +\beta}$ regularity results for the $\mA$-harmonic equation \eqref{A harmoninen} \cite{Manfredi}.

We follow the argument of \cite{LS_MRL} except for one part of the proof. In \cite{LS_MRL}, we used the fact if $\mA$ is $C^s$ in the $x$ variable where $s > 1$, then an $\mA$-harmonic function whose gradient is nonzero must be $C^{s+1}$. The assumption $s > 1$ was used to first show that the $\mA$-harmonic function is $W^{2,2}$ by a difference quotient argument, and then the $\mA$-harmonic equation was brought into nondivergence form where Schauder estimates could be applied. In this paper, we will work under the assumption $s > 0$ and apply Campanato type estimates and a frozen coefficient argument to the divergence form equation instead.

We recall the following classical result by Campanato. To that aim we denote the mean value of $u$ in a ball  $B(x_0,r)$ by $u_{x_0,r}$. 
The proof can be found in \cite{Giusti}.
\begin{prop}
\label{campanato}
Suppose that $u \in L^2(\Omega)$ satisfies
\[
\int_{B(x,r)} |u-u_{x,r}|^2 \leq C r^{n+ 2 \gamma} \qquad \text{for every }\, B(x,r) \subset \Omega, \,\, r \leq r_0,
\]
for some $\gamma \in (0,1)$ and $r_0>0$. Then $u \in C^\gamma(\Omega)$.
\end{prop}
We will also need the following standard algebraic lemma (see \cite[Lemma 7.3]{Giusti}).
\begin{lemma}
\label{lemma-algebra}
Let $\phi$ be a nonnegative and nondecreasing function on $[0,r_0]$ such that
\[
\phi(\rho) \leq A \left( \frac{\rho}{r}\right)^\alpha \phi(r) + Br^\beta \qquad \text{  for every }\, 0 <\rho \leq r  \leq r_0,
\]
where $A,B, \alpha, \beta$ are nonnegative constants such that $\beta < \alpha$. Then for every $\gamma \in (\beta, \alpha)$ there 
exists a constant $C$ such that  it holds
\[
\phi(\rho) \leq C \left( \left( \frac{\rho}{r}\right)^\gamma \phi(r) +  B\rho^\beta \right)  \qquad \text{  for every }\, 0 <\rho \leq r \leq r_0.
\]
\end{lemma}

The next result concerns $C^{1+\beta}$ regularity of $\mA$-harmonic functions, and requires the assumption that $\mA$ is $C^{\alpha}$ with respect to $x$ for some $\alpha > 0$. We recall that even in Euclidean space, solutions of the $p$-Laplace equation always have $C^{1+\beta}$ regularity for some $\beta > 0$, but $\beta$ could be very small and, in particular, solutions are in general not $C^2$ regular.

We write $B_r = B(0,r)$ for the open ball of radius $r > 0$ centered at the origin in $\mR^n$, and $[f]_{C^{\beta}} = \sup_{x \neq y} \frac{\abs{f(x)-f(y)}}{\abs{x-y}^{\beta}}$.

\begin{prop} \label{Aharmonic_regularity}
Suppose $u \in W^{1,p}(B(x_0,r))$ is a weak solution of 
\[
\mdiv \,\mA(x,\nabla u) = 0
\]
in $B(x_0,r)$ and  $\mA$ satisfies \eqref{mA_assumptions1}--\eqref{mA_assumptions2}. Then there are constants $C > 0$ and $\beta \in (0,1)$ only depending on $n$, $p$, $\alpha$, $\delta$ and  $M$ such that $u$ is $C^{1+\beta}$ regular in $B(x_0,r)$ and 
$$
\norm{\nabla u}_{L^{\infty}(B(x_0,r/2))} + r^{\beta} [\nabla u]_{C^{\beta}(B(x_0,r/2))} \leq C r^{-n/p} \norm{\nabla u}_{L^{p}(B(x_0,r))}.
$$
\end{prop}
\begin{proof}
We may assume that $x_0 = 0$ and $u$ solves $\mdiv\,\mA(x,\nabla u) = 0$ in $B_r$. Define scaled coordinates $\tilde{x} = x/r$ and a scaled function $\tilde{u}(\tilde{x}) = r^{-1} u(r \tilde{x})$. Then $\tilde{u} \in W^{1,p}(B_1)$ is a solution of 
\[
\mdiv\,\mA_r(\tilde{x},\nabla \tilde{u}(\tilde{x})) = 0 \quad \text{in } B_1
\]
where $\mA_r(\tilde{x},\tilde{\xi}) = \mA(r\tilde{x}, \tilde{\xi})$. Now $\mA_r$ satisfies \eqref{mA_assumptions1}--\eqref{mA_assumptions2} with the same constants $\delta$ and $M$ (this uses that $r \leq 1$).

We define 
\[
v = \frac{\tilde{u} - \tilde{u}_{0,1}}{\norm{\tilde{u} - \tilde{u}_{0,1}}_{W^{1,p}(B_1)}}.
\]
Then $v \in W^{1,p}(B_1)$ is a solution of $\mdiv\,\mA_r(\tilde{x}, \nabla v(\tilde{x})) = 0$ in $B_1$ with $\norm{v}_{W^{1,p}(B_1)} = 1$. It follows from \cite{Manfredi} (see also \cite{Serrin}, \cite{DiBenedetto})
 that there are $C > 0$ and $\beta > 0$, only depending on $n, p, \alpha, \delta$ and $M$, such that 
\[
\norm{v}_{C^{1+\beta}(B_{1/2})} \leq C.
\]
In particular, this implies that 
\[
\norm{\nabla v}_{L^{\infty}(B_{1/2})} + [\nabla v]_{C^{\beta}(B_{1/2})} \leq C.
\]
The Poincar\'e inequality yields $\norm{\tilde{u} - \tilde{u}_{0,1}}_{W^{1,p}(B_1)} \leq C \norm{\nabla \tilde{u}}_{L^p(B_1)}$. Consequently 
\[
\norm{\nabla \tilde{u}}_{L^{\infty}(B_{1/2})} +  [\nabla \tilde{u}]_{C^{\beta}(B_{1/2})} \leq C \norm{\nabla \tilde{u}}_{L^p(B_1)}.
\]
Undoing the scaling proves the result.
\end{proof}

The following lemma will be useful for the Schauder estimate that we  use later.
We will apply it in the case where either the solution of the $\mA$-harmonic equation or the solution of the equation with frozen coefficient
 has nonvanishing gradient, and $L^2$ estimates will be natural in this context.

\begin{lemma} \label{prop_two_equations_close}
Suppose $\mA$ satisfies \eqref{mA_assumptions1}--\eqref{mA_assumptions2} and   
$u_1, u_0 \in W^{1,p}(B(x_0,r))$, $r \leq 1$, are such that $u_1-u_0 \in  W^{1,p}_0(B(x_0,r))$ and $u_1$  is a weak solution  of
\[
\mdiv \,\mA(x,\nabla u_1) = 0
\]
and $u_0$ is a weak solution of the equation with coefficients frozen at $x_0$, 
\[
\mdiv \,\mA(x_0,\nabla u_0) = 0.
\]
If either $u_1$ or $u_0$  satisfies  the condition 
\begin{equation} \label{c_condition}
\text{$c_0 \leq \abs{\nabla u_i} \leq 1/c_0$ \, \, in $B(x_0,r)$}
\end{equation}
for some $c_0>0$, then we have the two estimates 
\begin{align}
\norm{\nabla u_1 - \nabla u_0}_{L^2(B(x_0,r/2))} &\leq C r^{\frac{n}{2}+\alpha}, \label{c_estimate1} \\
\norm{\nabla u_1 - \nabla u_0}_{L^{\infty}(B(x_0,r/4))} &\leq C  r^{\frac{2\alpha\beta}{n+2\beta}}. \label{c_estimate2}
\end{align}
The constants $C$ and $\beta$ depend on $c_0$, $n$, $p$, $\alpha$, $\delta$ and $M$.
\end{lemma}
\begin{proof}
Without loss of generality we may assume that $x_0 = 0$. We prove the result by assuming that $u_0$ satisfies \eqref{c_condition} since in the other case the argument is similar. Write $\mA_1 = \mA$ and $\mA_0(x,\xi) = \mA(x_0,\xi)$.

Let us first prove \eqref{c_estimate1}. We use the assumptions on $\mA$ and   the fact that we may use $u_1-u_0$ as a test function in the equations for $u_1$ and $u_0$ to conclude that 
\begin{align*}
I &:= \int_{B_r} (\abs{\nabla u_1} + \abs{\nabla u_0})^{p-2} \abs{\nabla u_1 - \nabla u_0}^2 \,dx \\
 &\leq \delta^{-1} \int_{B_r} ( \mA_1(x,\nabla u_1) - \mA_1(x,\nabla u_0) ) \cdot (\nabla u_1 - \nabla u_0) \,dx \\
 &= -\delta^{-1} \int_{B_r} ( \mA_1(x,\nabla u_0) - \mA_0(x,\nabla u_0) ) \cdot (\nabla u_1 - \nabla u_0) \,dx \\
  &\leq \delta^{-1} \left[ \int_{B_r} (\abs{\nabla u_1} + \abs{\nabla u_0})^{2-p} \abs{ \mA_1(x,\nabla u_0) - \mA_0(x,\nabla u_0)}^2 \,dx \right]^{1/2} I^{1/2}.
\end{align*}
Thus by \eqref{mA_assumptionsoneandhalf} 
\begin{align} \label{i_estimate}
I &\leq \delta^{-2} \int_{B_r} (\abs{\nabla u_1} + \abs{\nabla u_0})^{2-p} \abs{ \mA_1(x,\nabla u_0) - \mA_0(x,\nabla u_0)}^2 \notag \\
 &\leq \frac{M^2}{\delta^2} r^{2\alpha} \int_{B_r} (\abs{\nabla u_1} + \abs{\nabla u_0})^{2-p} \abs{\nabla u_0}^{2p-2}.
\end{align}

Assume first that $p \geq 2$. Then \eqref{c_condition} and  \eqref{i_estimate} imply 
\begin{align*}
\int_{B_{r/2}} \abs{\nabla u_1 - \nabla u_0}^2 \leq C  r^{2\alpha} \int_{B_r} \abs{\nabla u_0}^p  
\end{align*}
which is \eqref{c_estimate1}.

Let us  assume $1<p<2$. We estimate the right hand side of \eqref{i_estimate} by H\"older inequality with exponents $q = \frac{p}{2-p}$ and $q' = \frac{p}{2p-2}$ 
and obtain  
\begin{align}
 &\int_{B_r} (\abs{\nabla u_1} + \abs{\nabla u_0})^{2-p} \abs{\nabla u_0}^{2p-2} \notag \\
 &\leq \left( \int_{B_r} (\abs{\nabla u_1} + \abs{\nabla u_0})^p \right)^{\frac{2-p}{p}} \left( \int_{B_r} \abs{\nabla u_0}^p \right)^{\frac{2p-2}{p}} \notag \\
 &\leq (\norm{\nabla u_1}_{L^p(B_r)} + \norm{\nabla u_0}_{L^p(B_r)})^{2-p} \norm{\nabla u_0}_{L^p(B_r)}^{2p-2}. \label{b_estimate2}
\end{align}
Moreover, since $\int_{B_r} \mA_1(x,\nabla u_1) \cdot \nabla(u_1-u_0) \,dx = 0$, we have 
\begin{align*}
\delta \int_{B_r} \abs{\nabla u_1}^p \,dx &\leq \int_{B_r} \mA_1(x,\nabla u_1) \cdot \nabla u_1 \,dx = \int_{B_r} \mA_1(x,\nabla u_1) \cdot \nabla u_0 \,dx \\
 &\leq M \int_{B_r} \abs{\nabla u_1}^{p-1} \abs{\nabla u_0} \,dx \leq M \norm{\nabla u_1}_{L^p(B_r)}^{p-1} \norm{\nabla u_0}_{L^p(B_r)}.
\end{align*}
Hence
\begin{equation} \label{b_estimate3}
\norm{\nabla u_1}_{L^p(B_r)} \leq \frac{M}{\delta} \norm{\nabla u_0}_{L^p(B_r)}.
\end{equation}
Therefore the  the right hand side of \eqref{i_estimate} can be estimated by \eqref{b_estimate2} and  \eqref{b_estimate3}
\begin{equation} \label{b_estimate4}
 r^{2\alpha} \int_{B_r} (\abs{\nabla u_1} + \abs{\nabla u_0})^{2-p} \abs{\nabla u_0}^{2p-2} \leq C r^{2\alpha}  \int_{B_r} \abs{\nabla u_0}^p
\end{equation}
To estimate the left hand side of  \eqref{i_estimate} we use Proposition \ref{Aharmonic_regularity} and \eqref{b_estimate3} show that on $B_{r/2}$ 
\[
\begin{split}
\abs{\nabla u_0} + \abs{\nabla u_1} &\leq 1/c_0 + C r^{-n/p} \norm{\nabla u_1}_{L^p(B_r)} \\
&\leq 1/c_0 + C r^{-n/p} \norm{\nabla u_0}_{L^p(B_r)} \\
&\leq C.
\end{split}
\]
Therefore \eqref{i_estimate}  and  \eqref{b_estimate4} yield
\[
\int_{B_{r/2}} \abs{\nabla u_1 - \nabla u_0}^2  \leq C  r^{2\alpha} \int_{B_r} \abs{\nabla u_0}^p
\]
and  we again  have \eqref{c_estimate1}.

It remains to prove \eqref{c_estimate2} assuming \eqref{c_condition} for $u_0$. We first invoke Proposition \ref{Aharmonic_regularity} which shows that 
\[
r^{\beta} [\nabla u_j]_{C^{\beta}(B_{r/2})} \leq C r^{-n/p} \norm{\nabla u_j}_{L^{p}(B_{r})}, \quad j=0,1.
\]
Using \eqref{b_estimate3} and the fact that $\abs{\nabla u_0} \leq 1/c_0$ in $B_r$, we obtain 
\[
r^{\beta} [\nabla u_j]_{C^{\beta}(B_{r/2})} \leq C, \quad j=0,1.
\]
Combining this with \eqref{c_estimate1} and increasing the constant slightly gives the two estimates 
\begin{align*}
[r^{\beta}(\nabla u_1 - \nabla u_0)]_{C^{\beta}(B_{r/2})} &\leq C 4^{\frac{n}{2}+\beta}, \\
\norm{r^{\beta}(\nabla u_1 - \nabla u_0)}_{L^2(B_{r/2})} &\leq C  4^{\frac{n}{2}+\beta} \frac{r^{\frac{n}{2}+\alpha+\beta}}{4^{\frac{n}{2}+\beta}}.
\end{align*}
Interpolating these two estimates by \cite[Lemma A.1]{LS_MRL} (with the choices $\Omega = B_{r/2}$, $K = \overline{B_{r/4}}$, $M = C 4^{\frac{n}{2}+\beta}$ and $f = r^{\beta}(\nabla u_1 - \nabla u_0)$) gives 
\begin{align*}
 &\norm{r^{\beta}(\nabla u_1 - \nabla u_0)}_{L^{\infty}(B_{r/4})} \\
 &\quad \leq C_{n,p,\beta} (C 4^{\frac{n}{2}+\beta})^{\frac{n}{n+2\beta}} \norm{r^{\beta}(\nabla u_1 - \nabla u_0)}_{L^2(B_{r/2})}^{\frac{2\beta}{n+2\beta}} \\
 &\quad \leq C r^{\beta + \frac{2\alpha \beta}{n+2\beta}}.
\end{align*}
This proves \eqref{c_estimate2}.
\end{proof}

The next result shows that if $\mA$ is $C^{\alpha}$ with respect to $x$, then one can upgrade the $C^{1+\beta}$ regularity of a weak solution $u$ (recall that $\beta$ could be very small) to $C^{1+\alpha}$ regularity in any open set where $\nabla u$ is nonvanishing.

\begin{prop} \label{Aharmonic_regularity_nonvanishing}
Suppose $u \in W^{1,p}(B_1)$ is a weak solution of 
\[
\mdiv \,\mA(x,\nabla u) = 0
\]
and  $\mA$ satisfies \eqref{mA_assumptions1}--\eqref{mA_assumptions2}. 
If $\nabla u$ is nonvanishing in some open set $U \subset B_1$, then $u$ is $C^{1+\alpha}$ in $U$.
%If additionally $\mA(x,\xi)$ is $C^r$, $r>1$, with respect to $x$ variable and $C^{r+\e}$, for some $\e>0$, with respect to $\xi$ variable on $\R^n \setminus \{ 0 \}$, then $u$ is $C^{r+1}_*$ in any open set where $\nabla u$ is nonvanishing.
\end{prop}
\begin{proof}
Assume that $\nabla u$ is nonvanishing in $U \subset B_1$. Our aim is to show that $u$ is $C^{1+\alpha}$ near any $x_0 \in U$. Fix $x_0 \in U$ and choose $R > 0$ so that $B(x_0, 8 R) \subset U$. We will show that for some small $r_0 > 0$ one has 
\begin{equation}
\label{to show}
\int_{B(z,r)}|\nabla u-(\nabla u)_{z,r}|^2 \, dy \leq C r^{n+2 \alpha} \norm{\nabla u}_{L^2(B(x_0,8R))}^2
\end{equation}
whenever $B(z,r) \subset B(x_0,R)$ and $r \leq r_0$. This implies  $u \in C^{1+ \alpha}(B(x_0,R))$ by Proposition \ref{campanato}.

By Proposition \ref{Aharmonic_regularity}, $u$ is $C^{1+\beta}$ regular in $B_1$ and thus there is a constant $c_0 > 0$ such that 
\begin{equation} \label{u_bound_br}
c_0 \leq \abs{\nabla u} \leq 1/c_0 \quad \text{in } B(x_0,8R).
\end{equation}
Fix $z$ and $r$ with $B(z,r) \subset B(x_0,R)$, and let $v = v_r \in W^{1,p}(B(z,8r))$ be the weak solution of the frozen coefficient equation
\[
\mdiv \,\mA(z,\nabla v) = 0 \quad \text{in }\, B(z,8r), \qquad v = u \,\, \text{on }\, \partial B(z,8r).
\]
By Lemma \ref{prop_two_equations_close}  we have (with $C$ independent of $z$ and $r$) 
\begin{align}
\norm{\nabla u - \nabla v}_{L^2(B(z,4r))} &\leq C  r^{\frac{n}{2}+\alpha}, \label{vertailu standardi} \\
\norm{\nabla u - \nabla v}_{L^{\infty}(B(z,2r))} &\leq C r^{\frac{2\alpha\beta}{n+2\beta}}. \label{c_estimate2_linfty}
\end{align}
The bounds \eqref{u_bound_br} and \eqref{c_estimate2_linfty} imply  that there is $r_0>0$ depending on $c_0$ such that 
\begin{equation}
\label{bound below}
|\nabla v| \geq c_0/2 \qquad  \text{in } \,  B(z,2r)
\end{equation}
when $r \leq r_0$. Now $v$ is the solution of an $\mA$-harmonic equation in $B(z,2r)$ with constant coefficients  and $\nabla v$ is nonvanishing by \eqref{bound below}. A standard argument (see for instance \cite[Proposition 2.3]{LS_MRL}) implies that $v$ is $C^{\infty}$ in $\overline{B(z,r)}$. Thus arguing as in \cite[Theorem 7.7]{Giusti} we have the estimate
\begin{equation}
\label{camp for v}
\int_{B(z,\rho)} |\nabla v-(\nabla v)_{z,\rho}|^2 \, dx \leq C \left( \frac{\rho}{r}\right)^{n+2}\int_{B(z,r)} |\nabla v-(\nabla v)_{z,r}|^2   \, dx
\end{equation}
for every $\rho \leq r \leq r_0$ for some small $r_0$.  We use \eqref{vertailu standardi}, \eqref{camp for v} and the estimate $\int_{B(z,\rho)} \abs{f - f_{z,\rho}}^2 \leq \int_{B(z,\rho)} \abs{f-\lambda}^2$ for any $\lambda \in \mR$ to conclude that 
\[
\begin{split}
\int_{B(z,\rho)} &|\nabla u-(\nabla u)_{z,\rho}|^2 \, dx \\
&\leq 2\int_{B(z,\rho)} |\nabla v-(\nabla v)_{z,\rho}|^2 \, dx+ 2 \int_{B(z,\rho)} |\nabla (u-v)|^2 \, dx\\
&\leq C \left( \frac{\rho}{r}\right)^{n+2}\int_{B(z,r)} |\nabla v-(\nabla v)_{z,r}|^2   \, dx + 2 \int_{B(z,r)} |\nabla (u-v)|^2 \, dx \\ 
&\leq  C \left( \frac{\rho}{r}\right)^{n+2}\int_{B(z,r)} |\nabla u-(\nabla u)_{z,r}|^2   \, dx + C \int_{B(z,r)} |\nabla (u-v)|^2 \, dx\\
&\leq  C \left( \frac{\rho}{r}\right)^{n+2}\int_{B(z,r)} |\nabla u-(\nabla u)_{z,r}|^2   \, dx + Cr^{n+2 \alpha}
\end{split}
\]
for every $\rho \leq r \leq r_0$. The estimate \eqref{to show} then follows from Lemma \ref{lemma-algebra}.
\end{proof}

We continue to study the regularity properties of $\mA$-harmonic functions and prove a Calder\'on-Zygmund type result which will be needed later in Section~\ref{sec_weyl}. We denote the partial derivatives of $x \mapsto \mA(x, \xi)$ by $\partial_{x_i} \mA(x, \xi)$ and of  $\xi \mapsto \mA(x, \xi)$ by $\partial_{\xi_i} \mA(x, \xi)$.
 Let us assume that  in addition to the structural assumptions \eqref{mA_assumptions1}--\eqref{mA_assumptions2}  the vector field $\mA$ satisfies
\begin{equation}
\label{mA_assumption sobo}
|\partial_{x_i} \mA(x, \xi)| \leq f(x) |\xi|^{p-1}, \qquad \text{for some }\, f \in L^q(\Omega),
\end{equation}
where $q \geq 2$ and
\begin{equation}
\label{mA_assumption sobo 2}
(x,\xi) \mapsto \partial_{\xi_i} \mA(x, \xi) \,\,  \qquad \text{is continuous}.
\end{equation}

\begin{prop} \label{Aharmonic sobolev regularity}
Suppose $u \in W^{1,p}(B_1)$ is a weak solution of 
\[
\mdiv \,\mA(x,\nabla u) = 0 
\]
and  $\mA$ satisfies \eqref{mA_assumptions1}--\eqref{mA_assumptions2} and \eqref{mA_assumption sobo}, \eqref{mA_assumption sobo 2}. 
If $\nabla u$ is nonvanishing in some open set $U \subset B_1$, then $u \in W_{\text{loc}}^{2,q}(U)$.
\end{prop}
\begin{proof}
Assume that $\nabla u$ is nonvanishing in $B(x_0,2r) \subset B_1$. Without loss of generality we assume that $x_0 = 0$.  By the previous theorem we have that $u \in C^{1,\alpha}(B_{\frac{3r}{2}})$. 
We  show first that $u \in W^{2,2}(B_r)$. 

Let us fix a direction $e_i$ and choose a cutoff function $\zeta \in C_0^\infty(B_{\frac{4r}{3}})$ such that $0 \leq \zeta \leq 1$ and $\zeta \equiv 1$ in $\overline{B}_r$. For a small $h>0$ denote the difference quotient 
by $\Delta_h u(x) = \frac{1}{h}(u(x+he_i) -u(x))$. We use the test function $\varphi = \Delta_h u(x) \zeta^2$ and integrate against the equations  
\[
-\mdiv \,\mA(x,\nabla u(x)) = 0  \qquad \text{and} \qquad -\mdiv \,\mA(x+he_i,\nabla u(x+he_i)) = 0.
\]
By rearranging terms, this gives that 
\begin{equation}
\label{mA sobolemma 1}
\begin{split}
&\int_{B_{\frac{4r}{3}}} \left( \mA(x+ he_i,\nabla u(x+ he_i)) - \mA(x+ he_i,\nabla u(x))  \right) \cdot (\nabla \Delta_h u) \zeta^2 \, dx\\
&\leq  \int_{B_{\frac{4r}{3}}} \big|  \mA(x+ he_i,\nabla u(x)) - \mA(x,\nabla u(x)) \big| |\nabla \Delta_h u| \zeta^2 \, dx\\
&+ 2\int_{B_{\frac{4r}{3}}}  \big|\mA(x+ he_i,\nabla u(x+ he_i)) - \mA(x+ he_i,\nabla u(x))  \big| \, |\Delta_h u| \zeta |\nabla \zeta| \, dx\\
&+ 2\int_{B_{\frac{4r}{3}}}   \big|\mA(x+ he_i,\nabla u(x)) - \mA(x,\nabla u(x))  \big| \, |\Delta_h u| \zeta |\nabla \zeta| \, dx.
\end{split}
\end{equation}

We use the assumption \eqref{mA_assumptions2},  the $C^{1,\alpha}$ regularity of $u$ and  the assumption $|\nabla u| \geq c >0$ to  estimate  the left hand side of \eqref{mA sobolemma 1}
\begin{equation}
\label{mA sobolemma 2}
\begin{split}
\int_{B_{\frac{4r}{3}}} &\left( \mA(x+ he_i,\nabla u(x+ he_i)) - \mA(x+ he_i,\nabla u(x))  \right) \cdot (\nabla \Delta_h u) \zeta^2 \, dx\\
 &\geq c h\int_{B_{\frac{4r}{3}}}  |\nabla \Delta_h u |^2 \zeta^2 \, dx
\end{split}
\end{equation}
for some constant $c>0$. Next we use \eqref{mA_assumption sobo} and Young's inequality to estimate the first term on the right hand side of  \eqref{mA sobolemma 1}
\begin{equation}
\label{mA sobolemma 3}
\begin{split}
&\int_{B_{\frac{4r}{3}}} \big|  \mA(x+ he_i,\nabla u(x)) - \mA(x,\nabla u(x)) \big| |\nabla \Delta_h u| \zeta^2 \, dx\\
 &\leq \frac{ch}{4}\int_{B_{\frac{4r}{3}}}  |\nabla \Delta_h u |^2 \zeta^2 \, dx + \frac{C}{h} \int_{B_{\frac{4r}{3}}} \big|  \mA(x+ he_i,\nabla u(x)) - \mA(x,\nabla u(x)) \big|^2  \, dx\\
&\leq\frac{ch}{4}\int_{B_{\frac{4r}{3}}} |\nabla \Delta_h u |^2 \zeta^2 \, dx + \frac{C}{h} \int_{B_{\frac{4r}{3}}}\Big| \int_0^h \frac{\partial }{\partial t} \mA(x+ te_i,\nabla u(x)) \, dt \Big|^2  \, dx\\
&\leq\frac{ch}{4}\int_{B_{\frac{4r}{3}}} |\nabla \Delta_h u |^2 \zeta^2 \, dx + C \int_0^h \int_{B_{\frac{4r}{3}}}  \big| \partial_{x_i} \mA(x+ te_i,\nabla u(x)) \big|^2  \, dx dt \\
&\leq\frac{ch}{4}\int_{B_{\frac{4r}{3}}} |\nabla \Delta_h u |^2 \zeta^2 \, dx + Ch  \int_{B_{\frac{3r}{2}}} f^2  \, dx. 
\end{split}
\end{equation}
Similarly we have 
\begin{equation}
\label{mA sobolemma 4}
\begin{split}
\int_{B_{\frac{4r}{3}}}   &\big|\mA(x+ he_i,\nabla u(x)) - \mA(x,\nabla u(x))  \big| \, |\Delta_h u| \zeta |\nabla \zeta| \, dx\\
&\leq Ch  \int_{B_{\frac{3r}{2}}}   \left(|\nabla u|^2+ f^2\right)\, dx.
\end{split}
\end{equation}
Moreover by the assumption \eqref{mA_assumptionsoneandsome},  the $C^{1,\alpha}$ regularity of $u$ and  $|\nabla u| \geq c >0$ we have that
\begin{equation}
\label{mA sobolemma 5}
\begin{split}
\int_{B_{\frac{4r}{3}}}  &\big|\mA(x+ he_i,\nabla u(x+ he_i)) - \mA(x+ he_i,\nabla u(x))  \big| \, |\Delta_h u| \zeta |\nabla \zeta| \, dx\\
&\leq C \int_{B_{\frac{4r}{3}}}  | \nabla u(x+ he_i) - \nabla u(x)| \, |\Delta_h u| \zeta |\nabla \zeta| \, dx \\
&\leq\frac{ch}{4}\int_{B_{\frac{4r}{3}}} |\nabla \Delta_h u |^2 \zeta^2 \, dx  + Ch \int_{B_{\frac{3r}{2}}} |\nabla  u|^2 \, dx.
\end{split}
\end{equation}

We divide \eqref{mA sobolemma 1} by  $h$ and deduce from  \eqref{mA sobolemma 2}, \eqref{mA sobolemma 3}, \eqref{mA sobolemma 4} and \eqref{mA sobolemma 5} that 
\[
\int_{B_{r}} |\nabla \Delta_h u |^2 \, dx \leq C\int_{B_{\frac{3r}{2}}} \left(f^2\,  + |\nabla u|^2 \right)\, dx.
\]
By letting  $h \to 0$ we obtain that $u \in W^{2,2}(B_r)$.

We may thus differentiate the equation
\[
-\mdiv \,\mA(x,\nabla u(x)) = 0 
\]
with respect to $x_i$ and conclude that $v = \partial_{x_i} u$ is a solution of 
\begin{equation}
\label{mA_linearisoitu}
-\mdiv \left( \partial_{\xi_i} \mA(x,\nabla u(x)) \nabla v \right) = \mdiv \left( \partial_{x_i} \mA(x,\nabla u(x))  \right). 
\end{equation}
By the assumption \eqref{mA_assumptions2}  the linear operator on the left hand side of \eqref{mA_linearisoitu} is elliptic.  Moreover by \eqref{mA_assumption sobo 2} $x \mapsto \partial_{\xi_i} \mA(x,\nabla u(x)) $
is  continuous. On the other  hand  the vector field $x \mapsto  \partial_{x_i} \mA(x,\nabla u(x)) $ on the right hand side of \eqref{mA_linearisoitu}  is $L^q$ integrable and therefore by the classical Calder\'on-Zygmund estimate (also contained in Proposition \ref{prop_conealpha_regularity}(a)) we have that 
$v \in W^{1,q}$. This implies $u \in W^{2,q}(B_r)$.
\end{proof}

\begin{proof}[Proof of Theorem \ref{A-harm-coord}]
We may assume that $x_0 = 0$, $B_1 \subset \Omega$, and $S$ is the identity matrix (otherwise consider $(Sx)^j$ instead of $x^j$ below). The $\mA$-harmonic coordinates will be obtained as the map $U = (u^1, \ldots, u^n)$, where each $u^j = u^j_r$ solves for $r$ small the Dirichlet problem 
\[
\mdiv\,\mA(x,\nabla u^j) = 0 \   \   \text{in } B_r, \quad u^j|_{\partial B_r} = x^j.
\]
Let us fix $j \in \{1, \dots, n\}$. Then $v^j(x) = x^j$ is a solution of $\mdiv\,\mA(0,\nabla v^j) = 0$ in $B_r$
and 
\begin{equation*}
\abs{\nabla v^j} = 1 \text{ in } B_r.
\end{equation*}
We use Lemma  \ref{prop_two_equations_close} to conclude 
\begin{align*}
\norm{\nabla u^j - \nabla v^j}_{L^{\infty}(B_{r/4})} &\leq C r^{\frac{2\alpha\beta}{n+2\beta}}.
\end{align*}
Thus there is $r > 0$ such that $\abs{\nabla u^j-\nabla v^j} \leq \frac{1}{100}$ in $B_{r/4}$  and, in particular, $\nabla u^j$ is nonvanishing in $B_{r/4}$. Proposition \ref{Aharmonic_regularity_nonvanishing} shows that $u^j \in C^{1+\alpha}(B_{r/8})$.

Now each $u^j$ is $C^{1+\alpha}$ near $0$ and the Jacobian matrix of $U = (u^1, \ldots, u^n)$ is invertible at $0$. The inverse function theorem shows that $U$ is a $C^1$ diffeomorphism near $0$. Moreover $U$ is $C^{1+\alpha}$ and so is $U^{-1}$ since 
\begin{equation}\label{Inv_reg}
D(U^{-1})= (DU \circ U^{-1})^{-1}.
\end{equation}
The fact that any $C^1$ diffeomorphism whose coordinate functions are $\mA$-harmonic is a $C^{1+\alpha}$ diffeomorphism follows from Proposition \ref{Aharmonic_regularity_nonvanishing}.
\end{proof}

\begin{proof}[Proof of Theorem \ref{mainthm_pharmonic_coordinates}]
As discussed above, in any local coordinate system where the metric has $C^{\alpha}$ regularity, the $p$-harmonic equation is of the form $\mdiv\,\mA(x,\nabla u) = 0$ where $\mA$ is given by \eqref{A_def} and satisfies \eqref{mA_assumptions1}--\eqref{mA_assumptions2}. The existence and regularity of $p$-harmonic coordinates is then a consequence of Theorem \ref{A-harm-coord} for $s<1$. The case $s>1$ is contained in~\cite{LS_MRL}.

For $s=1$, we have $g\in W^{1,q}$ for all $q < \infty$ and thus by Proposition~\ref{Aharmonic sobolev regularity} the $p$-harmonic coordinates are $W^{2,q}$ regular. This regularity is enough for us to write the $p$-harmonic equation for each coordinate function $u = u^j$ in nondivergence form as 
\begin{equation*}
a^{jk} \partial_{jk} u=f
\end{equation*}
where the coefficients $a^{jk}$ and the function $f$ are defined by 
\begin{align*}
 a^{jk}(x) &=\p_{\xi_k}\mA^j(x,\nabla u(x)), \\
f(x) &=-\p_{x_j}\mA^j(x,\nabla u(x)).
\end{align*}
Now the linear equation $a^{jk} \partial_{jk} u = f$ is elliptic since $\nabla u$ is nonvanishing. Since $u\in W^{2,q}$ for all $q < \infty$, we have $\nabla u \in C^{\gamma}$ for any $\gamma < 1$ by the Sobolev embedding. Consequently $a^{jk}\in C^{\gamma}$ for $\gamma < 1$ and $f\in C^0$. Thus $u\in C^2_*$, see e.g.~\cite[Theorem 14.4.2]{Taylor3}. That the inverse of the $p$-harmonic coordinates are also of this regularity follows from the formula~\eqref{Inv_reg}. In deriving this conclusion, note that Zygmund spaces are closed under pointwise multiplication~\cite{Triebel}.

Suppose that the metric in original coordinates at $x_0$ is represented by the matrix $G_0$, and let $V$ be the corresponding $\mA$-harmonic coordinates. The metric in $p$-harmonic coordinates at $x_0$ is given by 
$$
DV(0)^{-t} G_0 DV(0)^{-1}.
$$
Choosing $V$ so that $DV(0)$ is very close to $G_0^{1/2}$, we can arrange that the metric in $p$-harmonic coordinates is arbitrarily close to identity at $x_0$.
\end{proof}

\begin{cor}\label{W1q}
Let $(M,g)$ be a Riemannian manifold with $g\in W^{1,q}$, $q>n$, in a local coordinate chart about some point $x_0 \in M$. There exists a $p$-harmonic coordinate system on a neighborhood of $x_0$ that is a local $W^{2,q}$ diffeomorphism. Moreover, the Riemannian metric is $W^{1,q}$ in these $p$-harmonic coordinates.
\end{cor}
\begin{proof}
Since $W^{1,q}\subset C^\alpha$ for some $\alpha >0$, there exists a $C^{1+\alpha}$ regular $p$-harmonic coordinate system $(v^i)$ on a neighborhood of $x_0$ by Theorem~\ref{mainthm_pharmonic_coordinates}. Let us denote $V=(v^1,\ldots, v^n)$. For  $g\in W^{1,q}$, the function $\mathcal{A}$ defining $p$-harmonic functions satisfies the conditions \eqref{mA_assumption sobo} and \eqref{mA_assumption sobo 2}, and thus by Proposition~\ref{Aharmonic sobolev regularity}, we have that $V \in W^{2,q}$. By the formula
 \[
D(V^{-1})= (DV\circ V^{-1})^{-1}
\]
it is easy to see (by using integration by substitution and the standard cofactor formula for an inverse matrix) that the same holds for the inverse of $V$. Note that $W^{1,q}$ is an algebra, see Section~\ref{sec_weyl}. 

That the Riemannian metric is $W^{1,q}$ in $p$-harmonic coordinates follows from the coordinate transformation rule for the Riemannian metric
$$
\tilde{g}=D(V^{-1})^T g|_{V^ {-1}} D(V^ {-1}).
$$
Here $g$ and $\tilde{g}$ are the metrics in the original and in the $p$-harmonic coordinates respectively.
\end{proof}

\section{Regularity of conformal and 1-quasiregular mappings} \label{sec_regularity_conformal}

The first application of $n$-harmonic coordinates is to the regularity of $C^1$ conformal mappings, or $W^{1,n}$ $1$-quasiregular mappings, between Riemannian manifolds with $C^s$, $s > 0$, $s\neq 1$, metric tensors. These results extend the corresponding results in~\cite{Iwaniec_thesis, LS_MRL} that covered H\"older exponents $s>1$, and the earlier results~\cite{Ferrand, Manfredi}. 

The proofs are essentially the same as in~\cite{LS_MRL} and are only sketched. We however cover a slight lapse in the proof of~\cite[Theorem 3.1]{LS_MRL} concerning the integer values $s=2,3,4,\ldots$ of the regularity parameter $s$. We also comment on the case $s=1$ after the theorem.

\begin{thm} \label{thm_c1_conformal_regularity}
Let $(U,g)$ and $(V,h)$ be Riemannian manifolds such that $g,h \in C^s$, $s > 0$, $s\neq 1$. If $\phi: U \to V$ is a $C^1$ diffeomorphism that is conformal in the sense that $\phi^* h = cg$ for some continuous positive function $c$, then $\phi$ is a $C^{s+1}$ diffeomorphism from $U$ onto $V$.
\end{thm}
\begin{proof}
Let first $g, h \in C^s$ with $s > 0$. Fix $x_0 \in U$, and let $v$ be $n$-harmonic coordinates near $\phi(x_0)$. Since $h \in C^s$, Theorem \ref{mainthm_pharmonic_coordinates} ensures that $v \in C^{s+1}_*$. Let $u = v \circ \phi$. Since $\phi$ is a $C^1$ conformal diffeomorphism, a straightforward calculation \cite[Theorem 3.1]{LS_MRL} shows that $u$ is an $n$-harmonic coordinate system near $x_0$. Now $g$ is $C^s$, which implies that $u$ is $C^{s+1}_*$ by Theorem \ref{mainthm_pharmonic_coordinates}. Thus $\phi = v^{-1}Ê\circ u$ near $x_0$ and $\phi^{-1} = u^{-1} \circ v$ near $\phi(x_0)$.

Assume now that $s > 0$ is not an integer. Then $u$, $v$ and their inverses are $C^{s+1}$, and thus $\phi$ and also $\phi^{-1}$ are $C^{s+1}$ (the spaces $C^{s+1} = C^{s+1}_*$ are closed under composition for $s > 0$ noninteger).

Assume next that $s \in \{2, 3, 4, \ldots \}$. We first show that the conformal factor $c$ in
\[
\phi^* h = cg
\]
is a $C^s$ function (this part of the proof is missing in~\cite[Theorem 3.1]{LS_MRL}). The argument above yields that $\phi \in C^{s+1-\eps}$ and $c\in C^{s-\eps}$ for any $\eps >0$.

Let us recall the transformation rules for the Ricci tensor and the scalar curvature under conformal scaling. If $\tilde{g}=e^{2f}g$, we have
\begin{align}\label{cnfcng}
R_{ij}(\tilde{g})&=R_{ij}(g)-(n-2)(\nabla_i\p_jf-(\p_if)(\p_jf))+(\Delta f-(n-2)|df|_g^2)g, \nonumber \\
R(\tilde{g})&=e^{-2f}(R(g)+2(n-1)\Delta f -(n-2)(n-1)|df|_g^2)
\end{align}
if $f$ and $g$ are sufficiently smooth, see e.g.~\cite{Besse}. Since $\phi$ is at least $C^{3-\eps}$ and $g,h$ are at least $C^2$, we have 
\[
e^{2f}:=c=g^{-1}\phi^*h\in C^{2-\eps} \text{ for any $\eps > 0$}.
\]
This, together with the assumption that $g,h\in C^2$, yield that the identities~\eqref{cnfcng} hold in the distributional sense as described later in Section~\ref{sec_weyl}. In addition to the identities above, the tensoriality identities
\begin{align*}
\phi^*R_{ij}(h)&=R_{ij}(\phi^*h), \\
 \phi^*R(h)&=R(\phi^*h)
\end{align*}
also hold in the distributional sense.

Let us now apply the scalar curvature to both sides of the equation
\[
\phi^*h=e^{2f}g.
\]
The identities above yield 
\[
\phi^* R(h) = e^{-2f}(R(g)+2(n-1)\Delta f -(n-2)(n-1)|df|_g^2).
\]
Since $g, h \in C^s$ and $\phi \in C^{s+1-\eps}$, $f \in C^{s-\eps}$, we may solve for $\Delta f$ which implies that $\Delta f \in C^{s-2}$. Now applying the Ricci tensor to $\phi^* h = e^{2f} g$, using the identities above and using that $\Delta f \in C^{s-2}$, gives that 
\begin{equation}\label{pp}
\p_{ij} f \in C^{s-2} \text{ for all $i,j=1,\ldots n$.}
\end{equation}
Thus $f$, and consequently $c$, is a $C^s$ function.

The rest of the proof of the theorem now follows by expressing the second derivatives of $\phi$ by the coordinate transformation formula 
\[
\partial_{i} \partial_{j} \phi^m=\Gamma_{ij}^k(cg) \partial_{k} \phi^m - \Gamma_{kl}^m(h)|_\phi \partial_{i} \phi^k \partial_{j} \phi^l
\]
of the Christoffel symbols (as in~\cite{Calabi, Taylor}). Since $c\in C^s$, the right hand side of this equation is in $C^{s-1}$ completing the proof.
\end{proof}

The reason we could not prove the theorem for $s=1$ is that in this case the relation \eqref{pp} for the second derivatives of the conformal factor would involve terms containing second distributional derivatives of $C^1$ functions. We could not verify that the conformal factor is in $C^1$ in this case.

The next result concerns the regularity of Riemannian $1$-quasiregular mappings and the regularity of conformal bi-Lipschitz mappings. We refer to~\cite{LS_MRL,dissertation} for details about 	Riemannian quasiregular mappings. %With the above discussion about the case $r=1$ in mind, 
The proof is the same as the proof of Theorem 4.4 in~\cite{LS_MRL}, except for the additional details for $s=2,3,4,\ldots$ given above. We omit the proof.

\begin{thm}\label{mainthm}
Let $(M,g)$ and $(N,h)$ be Riemannian manifolds, $n\geq 3$, with $g, h\in C^s$, $s>0$, $s\neq 1$. Let $\phi: M\rightarrow N$ be a non-constant mapping. Then the following are equivalent:
\begin{align*}
 &\phi  \mbox{ is a Riemannian $1$-quasiregular mapping}, \\ %\label{1qc_anal} \\
 &\phi  \mbox{ is locally bi-Lipschitz and } \phi^*h=c\,g \mbox{ a.e.,} \\ % \label{weak_form} \\ 
 &\phi  \mbox{ is a local $C^1$ diffeomorphism and } \phi^*h=c\,g, \\ %\label{1dif} \\
 &\phi  \mbox{ is a local $C^{s+1}$ diffeomorphism and } \phi^*h=c\,g. % \label{rdif}
\end{align*}
\end{thm}

Finally, we remark that the theorem holds also for $n=2$ if we a priori assume $\phi$ to be a local homeomorphism. This gives an extension of~\cite[Theorem 4.5]{LS_MRL} to H\"older exponents $s>0$, $s\neq 1$.

\section{Weyl tensor} \label{sec_weyl}

We now consider low regularity Riemannian metrics and establish an elliptic regularity result for the Weyl tensor. We first assume that 
$$
g_{jk} \in W^{1,2} \cap L^{\infty}
$$
in some system of local coordinates near a point. (In this section all function spaces are assumed to be of the local variety near a point, that is, we write $W^{1,2}$ instead of $W^{1,2}_{\mathrm{loc}}(U)$ etc.) This seems to be a minimal assumption for defining the Weyl tensor: the set $W^{1,2} \cap L^{\infty}$ is an algebra under pointwise multiplication \cite{KatoPonce}, \cite{BadrBernicotRuss}, and therefore $g^{jk}, \abs{g} \in W^{1,2} \cap L^{\infty}$.  We see that 
\begin{gather*}
\Gamma_{ab}^c = \frac{1}{2} g^{cr}(\partial_a g_{br} + \partial_b g_{ar} - \partial_r g_{ab}) \in L^2, \\
R_{abc}^{\phantom{abc}d} = \partial_a \Gamma_{bc}^d - \partial_b \Gamma_{ac}^d + \Gamma_{bc}^m \Gamma_{am}^d - \Gamma_{ac}^m \Gamma_{bm}^d \in W^{-1,2} + L^1, \\
R_{bc} = R_{abc}^{\phantom{abc}a} \in W^{-1,2} + L^1.
\end{gather*}
We write this schematically as 
\begin{align*}
\Gamma_{ab}^c &= O(g^{-1} \partial g), \\
R_{abc}^{\phantom{abc}d} &= O(\partial(g^{-1} \partial g) + (g^{-1} \partial g)^2), \\
R_{bc} &= O(\partial(g^{-1} \partial g) + (g^{-1} \partial g)^2)
\end{align*}
where $O(g^{-1} \partial g)$ denotes a finite linear combination of terms of the form $g^{pq} \partial_j g_{rs}$, etc. Now, since multiplication by an $W^{1,2} \cap L^{\infty}$ function maps $W^{1,q}$ to $W^{1,2}$ for any $q > n$, it also maps $W^{-1,2}$ to $W^{-1,q'}$, and we have 
\[
R_{abcd} = O \left(g \partial(g^{-1} \partial g) + g (g^{-1} \partial g)^2 \right) \in W^{-1,q'} + L^1.
\]

Since 
\[
P_{ab} = \frac{1}{n-2} \left[ R_{ab} - \frac{1}{2(n-1)} R g_{ab} \right]
\]
where $R = g^{rs} R_{rs}$ is the scalar curvature, we have 
\begin{align}
W_{abcd} &= R_{abcd} + P_{ac} g_{bd} - P_{bc} g_{ad} + P_{bd} g_{ac} - P_{ad} g_{bc} \notag \\
 &= (R_{abcd}+R_{ac}g_{bd}-R_{bc}g_{ad})+\l(\frac{1}{n-2}-1\r)(R_{ac}g_{bd}-R_{bc}g_{ad}) \notag \\
  &\qquad +\frac{1}{n-2}(R_{bd}g_{ac} -R_{ad}g_{bc})-\frac{R}{(n-1)(n-2)}(g_{ac}g_{bd}-g_{bc}g_{ad}) \notag \\
  &= O \left(g(1+g^{-1} g) \left[ \partial(g^{-1} \partial g) + (g^{-1} \partial g)^2 \right] \right) \in W^{-1,q'} + L^1. \label{weyl_schematic}
\end{align}
This shows that if $g_{jk} \in W^{1,2} \cap L^{\infty}$ in some system of local coordinates, then one can make sense of the components $W_{abcd}$ of the Weyl tensor as elements in $W^{-1,q'} + L^1$ in these coordinates for any $q > n$. It is easy to see that the identity $W(cg) = cW(g)$ remains true if $c, g_{jk} \in W^{1,2} \cap L^{\infty}$. 

We wish to understand the equation $W_{abcd}(g) = 0$ as a quasilinear divergence form system for the matrix elements $g_{jk}$. Below we write $D = -i\nabla$.

\begin{lemma} \label{lemma_weyl_divergenceform}
Let $g_{jk} \in W^{1,2} \cap L^{\infty}$ in some local coordinates. If $W_{abcd}(g) = 0$ in these coordinates and if $u \in W^{1,2} \cap L^{\infty}$ is the column vector consisting of the matrix elements $g_{jk}$, then $u$ is a distributional solution of 
\[
D_l(A^{lm}(u) D_m u) = B(u, Du)
\]
where $B$ is given by 
\[
B_a(u, Du) = B_a^{lm}(u) D_l u \cdot D_m u, \qquad 1 \leq a \leq M,
\]
and each $A^{lm}$ (resp.\ $B_a^{lm}$) is an $M \times N$ (resp.\ $N \times N$) matrix function whose entries are rational functions of the components of $u$. Moreover, each $A^{lm}$ and $B_a^{lm}$ is smooth in the set of vectors $u$ corresponding to positive definite symmetric matrices. One can choose $M = n^4$ and $N = \frac{n(n+1)}{2}$.
\end{lemma}
\begin{proof}
Follows from \eqref{weyl_schematic} after using the Leibniz rule: 
\begin{align*}
W_{abcd} &= O ( \partial(g(1+g^{-1} g) g^{-1} \partial g) - \partial(g(1+g^{-1} g)) g^{-1} \partial g \\
 &\qquad \quad + g(1+g^{-1} g)(g^{-1} \partial g)^2 ).
\end{align*}
The fact that $A^{lm}$ and $B_a^{lm}$ are rational functions follows by using Cramer's rule for the inverse matrix $(g^{jk})$.
\end{proof}

The main point, as observed in \cite{LS_Bach}, is that the equation $W_{abcd}(g) = 0$ becomes elliptic if it is written in $n$-harmonic coordinates and the determinant is normalized to one. This is expressed by the next result.

\begin{lemma} \label{lemma_weyl_elliptic}
Let $g_{jk} \in W^{1,q} \cap C^{\alpha}$ in some local coordinates, where $q \geq 2$ and $0 < \alpha < 1$. If $g_{ab}$ is the representation of $g$ in any $n$-harmonic coordinate system, and if 
\[
\hat{g}_{ab} = \abs{g}^{-1/n} g_{ab},
\]
then $\hat{g}_{ab} \in W^{1,q} \cap C^{\alpha}$. If $\hat{u}$ is the column vector consisting of the matrix elements $\hat{g}_{ab}$, and if $\hat{A}^{lm}(x) = A^{lm}(\hat{u}(x))$, then the linear operator 
\[
v \mapsto D_l( \hat{A}^{lm} D_m v)
\]
is overdetermined elliptic (i.e.\ its principal symbol is injective).
\end{lemma}
\begin{proof}
Since $g_{jk} \in W^{1,q} \cap C^{\alpha}$, by Theorem \ref{mainthm_pharmonic_coordinates} and Proposition \ref{Aharmonic sobolev regularity} we know that any $n$-harmonic coordinate system is $W^{2,q} \cap C^{1+\alpha}$. Now $W^{1,q} \cap C^{\alpha}$ is an algebra under pointwise multiplication, which shows that $g_{ab}$ and $\hat{g}_{ab}$ are in $W^{1,q} \cap C^{\alpha}$.

To prove ellipticity, we need to show that the matrix valued symbol $\hat{A}^{lm} \xi_l \xi_m$ is injective for any $\xi \in \mR^n \setminus \{0\}$. In order to simplify notation we will omit all hats in the rest of this proof and write $g_{ab}$ instead of $\hat{g}_{ab}$ etc. Assuming for the moment that $g$ is $C^{\infty}$, we note the following consequences of Bianchi identities, 
\[
\nabla^d M_{abcd} = 0, \qquad \nabla^a N_{abcd} = 0
\]
where 
\[
M_{abcd} = R_{abcd}+R_{ac}g_{bd}-R_{bc}g_{ad}, \qquad N_{abcd} = R_{ab}g_{cd}-\frac{1}{2}Rg_{ab}g_{cd}.
\]
The left hand sides of these Bianchi type identities may be viewed as third order nonlinear operators acting on $g$. After linearization we obtain the following symbol versions of the previous identities, 
\begin{equation} \label{bianchi_principalsymbol_identities}
\xi^d \sigma_\xi( M_{abcd}) = 0, \qquad \xi^a \sigma_\xi(N_{abcd}) = 0.
\end{equation}
Here $\sigma_\xi(P) = \sigma_{\xi}(P)(g)h$ is the principal symbol of the linearization of an operator $P$, linearized at $g = g_{ab}$ and acting on some $h = h_{ab}$. Now \eqref{bianchi_principalsymbol_identities} are algebraic equations for the principal symbols $\sigma_\xi( M_{abcd})$ and $\sigma_\xi(N_{abcd})$ when $g \in C^{\infty}$, but the since these principal symbols can also be computed for $W^{1,q} \cap C^{\alpha}$ metrics and their algebraic expressions remain identical, it follows that \eqref{bianchi_principalsymbol_identities} is valid for $g_{ab} \in W^{1,q} \cap C^{\alpha}$.
%These identities hold since $W^{1,q} \cap C^{\alpha}$ regularity for $g_{ab}$ is enough for us to calculate the linearizations of the curvature tensors in question. Note also that calculating the principal symbol is a pointwise operation making the calculations of the symbol identities above purely algebraic computations.
 
By \eqref{weyl_schematic} we have 
\begin{multline*}
 W_{abcd} = M_{abcd} - \frac{n-3}{n-2} N_{acbd} +\frac{1}{n-2}(N_{bdac}-N_{adbc}) \\
  +\frac{n-3}{n-2}R_{bc}g_{ad}-\frac{n-3}{2(n-2)}Rg_{ac}g_{bd} + \frac{1}{2(n-2)}(Rg_{bd}g_{ac}-Rg_{ad}g_{bc}) \\
  -\frac{R}{(n-1)(n-2)}(g_{ac}g_{bd}-g_{bc}g_{ad}).
\end{multline*}
Let us now assume that $\sigma_\xi(W_{abcd})h=0$, where $h=(h_{ij})$. Contracting this equation by $\xi^a \xi^d$ and using \eqref{bianchi_principalsymbol_identities} yields
\begin{multline*}
0=\frac{2(n-2)}{3-n}\xi^a\xi^d\sigma_\xi(W_{abcd})h =-2|\xi|^2\sigma_\xi(R_{bc})h+\frac{n-2}{n-1}\xi_b\xi_c\sigma_\xi(R)h \\
+\frac{1}{n-1}|\xi|^2 g_{bc} \sigma_\xi(R)h. \\
\end{multline*}
The symbol on the right hand side is now (a scalar multiple of) the symbol of the Bach tensor, see \cite[Lemma 3.2]{LS_Bach}. The $W^{1,q} \cap C^{\alpha}$ regularity is also enough for the $n$-harmonic coordinate gauge condition in \cite{LS_Bach} for the contracted Christoffel symbols (with $g$ having determinant one) to be written as
 \[
 \Gamma^k(g)=\frac{n-2}{2}g^{ka}\p_a \mbox{log}\,g^{kk}.
 \]
The required ellipticity follows now from \cite[Lemma 2.3]{LS_Bach}.
\end{proof}

\begin{proof}[Proof of Theorem~\ref{mainthm_weyl_flatness}]
The assumption that $g \in W^{1,n} \cap C^{\alpha}$ in some local coordinates together with Theorem \ref{mainthm_pharmonic_coordinates} implies that there exists an $n$-harmonic coordinate system with $C^{1+\alpha}$ regularity. By Lemma \ref{lemma_weyl_elliptic} we know that in any $n$-harmonic coordinates, $g_{ab}$ and $\hat{g}_{ab} = \abs{g}^{-1/n} g_{ab}$ are in $W^{1,n} \cap C^{\alpha}$. We will prove that 
\begin{equation} \label{gabhat_smooth}
\hat{g}_{ab} \in C^{\infty}.
\end{equation}
Since the Weyl tensor of $\hat{g}_{ab}$ vanishes, we get that $\hat{g}_{ab} = c_0 \delta_{ab}$ where $c_0 \in C^{\infty}$. Thus $g_{ab} = \abs{g}^{1/n} c_0 \delta_{ab}$ where $c = \abs{g}^{1/n} c_0$ is a positive function in $W^{1,n} \cap C^{\alpha}$ as required.

It remains to prove \eqref{gabhat_smooth}. In the rest of the proof, we will omit hats and write $g_{ab}$ and $u$ instead of $\hat{g}_{ab}$ and $\hat{u}$ etc. Since the Weyl tensor of $g$ vanishes, Lemmas \ref{lemma_weyl_divergenceform}--\ref{lemma_weyl_elliptic} imply that $u$ solves the overdetermined elliptic system 
\[
D_l(A^{lm}(u) D_m u) = B(u, Du)
\]
where $u \in W^{1,n} \cap C^{\alpha}$ is the column vector consisting of the matrix elements of $g_{ab}$. We claim that for any $t \geq n$, one has 
\begin{equation} \label{wont_claim}
\left\{ \begin{array}{c} D_l(A^{lm}(u) D_m u) = B(u, Du) \\[2pt] u \in W^{1,t} \cap C^{\alpha} \end{array} \right. \implies u \in W^{1,t+\delta} \cap C^{\alpha}
\end{equation}
where $\delta = \frac{\alpha}{1-\alpha} \frac{n}{4}$. If \eqref{wont_claim} is known, then starting from $u \in W^{1,n} \cap C^{\alpha}$ and iterating \eqref{wont_claim} gives $u \in W^{1,q}$ for any $q < \infty$. Then by Sobolev embedding $u \in C^{\beta}$ for any $\beta < 1$, and $u$ solves 
\[
D_l(A^{lm}(u) D_m u) = B(u, Du)
\]
where $A^{lm}(u) \in C^{\beta}$ and $B(u,Du) \in L^q$ for all $\beta < 1$ and $q < \infty$. The Schauder estimates in Proposition \ref{prop_conealpha_regularity}(c) give that $u \in C^{1+\beta}$ for all $\beta < 1$, and iterating the higher order Schauder estimates in Proposition \ref{prop_conealpha_regularity}(d) gives that $u \in C^{\infty}$. This shows \eqref{gabhat_smooth}.

To prove \eqref{wont_claim}, note that if $u \in W^{1,t} \cap C^{\alpha}$ and $t \geq n$, then $x \mapsto A^{lm}(u(x))$ and $x \mapsto B_a^{lm}(u(x))$ are in $W^{1,t} \cap C^{\alpha}$ (this set is an algebra under pointwise multiplication). Thus, if we write $f(x) = B(u(x), Du(x))$, we see that $u$ solves 
\[
D_l(A^{lm}(u) D_m u) = f
\]
with $A^{lm}(u) \in W^{1,t} \cap C^{\alpha}$ and $f \in L^{t/2}$. The Calder\'on-Zygmund estimate in Proposition \ref{prop_conealpha_regularity}(b) implies that $u \in W^{2,t/2}$. But we also know that $u \in C^{\alpha}$, and $C^{\alpha} \subset W^{\eps,q}$ whenever $\eps < \alpha$ and $q < \infty$. Under these conditions we have 
\[
u \in W^{2,t/2} \cap W^{\eps,q},
\]
and by complex interpolation $u \in [W^{2,t/2}, W^{\eps,q}]_{\theta}$ for $0 < \theta < 1$. We choose $\theta$ so that 
\[
(1-\theta) \cdot 2 + \theta \cdot \eps = 1,
\]
which gives $\theta = \frac{1}{2-\eps}$. Thus $u \in W^{1,p}$ where 
\[
\frac{1}{p} = (1-\theta) \frac{2}{t} + \theta \frac{1}{q} = \frac{1-\eps}{2-\eps} \frac{2}{t} + \theta \frac{1}{q}.
\]
Since $q < \infty$ was arbitrary, we get $u \in W^{1,p}$ whenever $p < \frac{2-\eps}{1-\eps} \frac{t}{2}$ and $\eps < \alpha$. In particular this gives \eqref{wont_claim}.
\end{proof}

\appendix

\section{Regularity for overdetermined elliptic systems} \label{sec_appendix_regularity}

Let $\Omega \subset \mR^n$, $n \geq 2$, be a bounded open set. Let us consider an $M \times N$ linear system of PDEs $P(x,D) u = f$ in $\Omega$, where $D = -i\nabla$ and 
\begin{equation} \label{p_condition_1}
P(x,D) u = D_l(A^{lm} D_m u)
\end{equation}
and the coefficients satisfy 
\begin{equation} \label{p_condition_2}
A^{lm} \in C^0_{\mathrm{loc}}(\Omega, \mR^{M \times N}).
\end{equation}
Let $p(x,\xi) = A^{lm}(x) \xi_l \xi_m$ be the principal symbol of $P$. We assume that $P$ is overdetermined elliptic in the sense that for any $x \in \Omega$ and $\xi \in \mR^n \setminus \{0\}$, the matrix $p(x,\xi)$ is injective:
\begin{equation} \label{p_condition_3}
p(x,\xi) \zeta = 0, \ \  \zeta \in \mR^N \implies \zeta = 0.
\end{equation}
If this holds, then necessarily $M \geq N$.

We will prove the following Calder\'on-Zygmund and Schauder regularity results for overdetermined elliptic systems in divergence form.

\begin{prop} \label{prop_conealpha_regularity}
Let $u \in W^{1,q_0}_{\mathrm{loc}}(\Omega, \mR^N)$ solve $P(x,D)u = f$ in the sense of distributions in $\Omega$, where $P$ satisfies \eqref{p_condition_1}--\eqref{p_condition_3} and $q_0 > 1$.
\begin{enumerate}
\item[(a)]
If $1 < q < \infty$ and 
\[
A^{lm} \in C^0_{\mathrm{loc}}, \qquad f \in W^{-1,q}_{\mathrm{loc}},
\]
then $u \in W^{1,q}_{\mathrm{loc}}$.
\item[(b)]
If $1 < q < \infty$ and 
\[
A^{lm} \in C^0_{\mathrm{loc}} \cap W^{1,t}_{\mathrm{loc}}, \qquad f \in L^{q}_{\mathrm{loc}},
\]
where 
\[
\left\{ \begin{array}{cl} t = \max\{n,q\}, & q \neq n, \\ t > n, & q = n, \end{array} \right. 
\]
then $u \in W^{2,q}_{\mathrm{loc}}$.
\item[(c)] 
If $0 < \alpha < 1$ and 
\[
A^{lm} \in C^{\alpha}_{\mathrm{loc}}, \qquad f \in C^{-1+\alpha}_{\mathrm{loc}},
\]
then $u \in C^{1+\alpha}_{\mathrm{loc}}$.
\item[(d)] 
If $k \geq 0$ is an integer, $0 < \alpha < 1$, and 
\[
A^{lm} \in C^{k+\alpha}_{\mathrm{loc}}, \qquad f \in C^{k-1+\alpha}_{\mathrm{loc}},
\]
then $u \in C^{k+1+\alpha}_{\mathrm{loc}}$.
\end{enumerate}
\end{prop}

\begin{remark}
As examples of functions in $W^{-1,q}_{\mathrm{loc}}$ and $C^{-1+\alpha}_{\mathrm{loc}}$, we mention that if $1 < q < \infty$, one has 
\[
f \in L^s_{\mathrm{loc}} \text{ where } \left\{ \begin{array}{cl} s=1, & q < \frac{n}{n-1}, \\ s>1, & q = \frac{n}{n-1}, \\ s=\frac{nq}{n+q},& q > \frac{n}{n-1}, \end{array} \right. \ \  f^j \in L^q_{\mathrm{loc}} \implies f + D_j f^j \in W^{-1,q}_{\mathrm{loc}},
\]
and if $0 < \alpha < 1$, one has 
\[
f \in L^{\frac{n}{1-\alpha}}_{\mathrm{loc}}, \quad f^j \in C^{\alpha}_{\mathrm{loc}} \implies f + D_j f^j \in C^{-1+\alpha}_{\mathrm{loc}}.
\]
\end{remark}

\begin{remark}
It is well known that for any $q > 2$ there is a uniformly elliptic scalar equation $D_j(a^{jk} D_k u) = 0$ where $a^{jk}$ are $L^{\infty}$, such that there is a nontrivial $W^{1,2}$ solution that is not in $W^{1,q}$ \cite{Meyers}. Thus (a) above does not hold if $A^{lm} \in L^{\infty}$, but (a) may remain true if one assumes $A^{lm} \in L^{\infty} \cap VMO$ (see \cite{DiFazio} for the scalar case).
\end{remark}

We will need some standard facts about function spaces, see \cite{BerghLofstrom}, \cite{Taylor3} for more details. If $s \in \mR$ and $1 < p < \infty$, then $W^{s,p} = W^{s,p}(\mR^n)$ is the Banach space with norm 
\[
\norm{f}_{W^{s,p}(\mR^n)} = \norm{\br{D}^s f}_{L^p(\mR^n)}.
\]
Here, $\br{\xi} = (1+\abs{\xi}^2)^{1/2}$ and $\br{D}^s$ is the Fourier multiplier corresponding to $\br{\xi}^s$. In general, if $m(\xi)$ is a $C^{\infty}$ function which is polynomially bounded in $\mR^n$ together with its derivatives, the corresponding Fourier multiplier $m(D)$ is the map on $\mS'(\mR^n)$ (the space of tempered distributions) defined by 
\[
m(D)f = \mathscr{F}^{-1}\{ m(\xi) \hat{f}(\xi) \}.
\]
If $k \in \{ 0, 1, 2, \ldots\}$ and $0 < \alpha < 1$, then $C^{k+\alpha} = C^{k+\alpha}(\mR^n)$ is the Banach space with norm 
\[
\norm{f}_{C^{k+\alpha}(\mR^n)} = \sum_{\abs{\gamma} \leq k} \norm{D^{\gamma} f}_{L^{\infty}(\mR^n)} + \sum_{\abs{\gamma} = k} \sup_{x \neq y} \frac{\abs{D^{\gamma} f(x)-D^{\gamma} f(y)}}{\abs{x-y}^{\alpha}}.
\]
One has $C^{k+\alpha} = C^{k+\alpha}_{*} = B^{k+\alpha}_{\infty \infty}$ where $C^s_*$ are the Zygmund spaces and $B^s_{pq}$ the Besov spaces in $\mR^n$. We write $C^{-1+\alpha} = C^{-1+\alpha}_*$ for $0 < \alpha < 1$. The operator $\br{D}^t$ is an isomorphism $W^{s,p} \to W^{s-t,p}$ and $C^s_{*} \to C^{s-t}_{*}$, and $D_l$ is a bounded map $W^{s,p} \to W^{s-1,p}$ and $C^s_* \to C^{s-1}_*$ for any $s, t \in \mR$ and $1 < p < \infty$ \cite[Section 13]{Taylor3}. We will also use the Sobolev embeddings 
\begin{equation*}
W^{1,p} \subset \left\{ \begin{array}{cc} L^{p^*}, & 1 \leq p < n, \\ \bigcap_{q < \infty} L^q_{\mathrm{loc}}, & p = n, \\ C^{1-n/p}, & n < p < \infty, \end{array} \right.
\end{equation*}
where $p^* = \frac{np}{n-p}$ for $1 \leq p < n$.

The first step in the proof of Proposition \ref{prop_conealpha_regularity} is a parametrix construction for solutions of $P_{x_0} v = f$ in $\mR^n$, where $P_{x_0}$ the frozen coefficient operator 
\[
P_{x_0} = A^{lm}(x_0) D_{lm}, \qquad x_0 \in \Omega.
\]
We fix a cutoff function $\psi \in C^{\infty}_c(\mR^n)$ with $\psi = 1$ for $\abs{\xi} \leq 1/2$ and $\psi = 0$ for $\abs{\xi} \geq 1$.

\begin{lemma} \label{lemma_v_representation_formula}
Fix $x_0 \in \Omega$. There is an operator $E_{x_0} = e_{x_0}(D)$ such that for any $v \in \mS'(\mR^n, \mR^N)$ one has 
\begin{equation} \label{v_representation_formula}
v = E_{x_0} P_{x_0} v + \psi(D) v
\end{equation}
where $e_{x_0}: \mR^n \to \mR^{N \times M}$ is the bounded smooth matrix valued function 
\[
e_{x_0}(\xi) = (1-\psi(\xi)) (p(x_0,\xi)^* p(x_0,\xi))^{-1} p(x_0,\xi)^*.
\]
The operator $E_{x_0}$ is a bounded map $W^{s,p}(\mR^n, \mR^M) \to W^{s+2,p}(\mR^n,\mR^N)$ and $C^s_{*}(\mR^n, \mR^M) \to C^{s+2}_{*}(\mR^n,\mR^N)$ whenever $s \in \mR$ and $1 < p < \infty$.
\end{lemma}
\begin{proof}
The operator $P_{x_0}$ has full symbol $p(x_0,\xi) = A^{lm}(x_0) \xi_l \xi_m$. Then 
\[
p(x_0,\xi)^* p(x_0,\xi) = \sum_{j,k,l,m=1}^n A^{jk}(x_0)^* A^{lm}(x_0) \xi_j \xi_k \xi_l \xi_m.
\]
The ellipticity condition \eqref{p_condition_3} implies that 
\[
\det{p(x_0,\xi)^* p(x_0,\xi)} \neq 0, \quad \xi \in \mR^n \setminus \{0\}.
\]
Since $p(x_0,\xi)^* p(x_0,\xi)$ is homogeneous of degree $4$ in $\xi$, the function $e_{x_0}$ defined in the statement of the lemma is smooth in $\xi$ and homogeneous of degree $-2$ for $\abs{\xi}$ large.

If $v \in \mS'(\mR^n, \mR^N)$, the identity 
\[
\hat{v}(\xi) = (1-\psi(\xi)) \hat{v}(\xi) + \psi(\xi) \hat{v}(\xi) = e_{x_0}(\xi) p(x_0,\xi) \hat{v}(\xi) + \psi(\xi) \hat{v}(\xi)
\]
implies $v = E_{x_0} P_{x_0} v + \psi(D) v$ by taking inverse Fourier transforms. Now $e_{x_0}(\xi)$ is a matrix valued classical symbol in the class $S^{-2}_{1,0}$, and the mapping properties of the corresponding pseudodifferential operator \cite[Section 13]{Taylor3} imply the mapping properties for $E_{x_0}$. (Alternatively, one could apply the vector valued Mihlin multiplier theorem \cite[Theorem 6.1.6]{BerghLofstrom} and its version on Besov spaces to $m_{x_0} = \br{\xi}^2 e_{x_0}$.)
\end{proof}

Next we give a local representation formula for solutions of the equation $P(x,D)u = f$. To prove the Schauder estimates it is convenient to do a rescaling: if $x_0 \in \Omega$ and $r > 0$ is small, we write 
\[
\tilde{A}^{lm}(y) = A^{lm}(x_0+ry), \quad \tilde{u}(y) = u(x_0+ry), \quad \tilde{f}(y) = f(x_0+ry).
\]
We also fix a cutoff function $\chi \in C^{\infty}_c(\mR^n)$ with $\chi = 1$ for $\abs{x} \leq 1/2$ and $\mathrm{supp}(\chi) \subset B(0,3/4)$, and we let $\chi_2(y) = \chi(y/2)$.

\begin{lemma} \label{lemma_local_representation_formula}
Let $u \in W^{1,q_0}_{\mathrm{loc}}(\Omega, \mR^N)$ solve $P(x,D)u = f $ in $\Omega$, where $q_0 > 1$. Fix a point $x_0 \in \Omega$, and assume that $B(x_0,2r) \subset \Omega$. Then $v = \chi \tilde{u}$ satisfies the equation 
\[
v + E_{x_0} Q_{x_0} v = E_{x_0} F + \psi(D) \chi_{2} v \qquad \text{in $\mR^n$},
\]
where 
\[
Q_{x_0} w = D_l (\chi_2(\tilde{A}^{lm}-\tilde{A}^{lm}(0)) D_m w)
\]
and 
\begin{align*}
F = r^2 \chi \tilde{f} + \tilde{A}^{lm} D_l \chi D_m \tilde{u} + D_l(\tilde{A}^{lm} \tilde{u} D_m \chi) .
\end{align*}
\end{lemma}
\begin{proof}
Note that $\tilde{u}$ satisfies the equation 
\[
D_l (\tilde{A}^{lm} D_m \tilde{u}) = r^2 \tilde{f} \quad \text{in } B(0,2).
\]
If $v = \chi \tilde{u}$, we compute 
\begin{align*}
D_l (\tilde{A}^{lm} D_m v) &= D_l (\tilde{A}^{lm} D_m (\chi \tilde{u})) = D_l (\tilde{A}^{lm} \chi D_m \tilde{u} + \tilde{A}^{lm} \tilde{u} D_m \chi) \\
 &= \chi(r^2 \tilde{f}) + \tilde{A}^{lm} D_l \chi D_m \tilde{u} + D_l(\tilde{A}^{lm} \tilde{u} D_m \chi).
\end{align*}
Note that the left hand side operator frozen at $0$ is just $P_{x_0}$. Thus 
\[
D_l (\tilde{A}^{lm} D_m v) = P_{x_0} v + Q_{x_0} v
\]
where $Q_{x_0}$ is as stated. It follows that 
\[
P_{x_0} v + Q_{x_0} v = F \quad \text{in $\mR^n$}.
\]
The result follows from Lemma \ref{lemma_v_representation_formula}.
\end{proof}

\begin{proof}[Proof of Proposition \ref{prop_conealpha_regularity}]
Let $u \in W^{1,q_0}_{\mathrm{loc}}(\Omega, \mR^N)$ solve $P(x,D)u = f $ in $\Omega$, where $q_0 > 1$. Fix $x_0 \in \Omega$, and assume that $B(x_0,2r) \subset \Omega$. By Lemma \ref{lemma_local_representation_formula}, the function $v = \chi \tilde{u}$ satisfies 
\[
T(v) = G \quad \text{in $\mR^n$},
\]
where 
\begin{align*}
T(w) &= w + E_{x_0} Q_{x_0} w, \\
G &= E_{x_0} F + \psi(D) \chi_{2} v.
\end{align*}

Assume now the conditions in (a). We may assume $q_0 < n$, so $W^{1,q_0} \subset L^{q_0^*}$ by Sobolev embedding. We may also assume $q_0 < q$ (otherwise there is nothing to prove). Fix $p_0 > 2$ so that $q_0, q \in [p_0', p_0]$. We claim that if $r$ is chosen small enough, then 
\begin{gather}
\smash{\raisebox{-\dimexpr.5\normalbaselineskip+.5\jot}{$%
            \left\{ \begin{array}{@{}c@{}}\\[\jot]\\[\jot]\end{array}\right.$}}
\text{$T$ is invertible on $W^{1,p}(\mR^n, \mR^N)$ for $p \in [p_0', p_0]$, and }\label{tboundedinvertible} \\
\text{$G \in W^{1,p}$ for $p \in [q_0, \min\{q_0^*,q\}]$.} \label{gwoneq}
\end{gather}
Assuming these two claims, it follows that $v = T^{-1}(G) \in W^{1,\min\{q_0^*,q\}}$. If $q \leq q_0^*$ we have $v \in W^{1,q}$ and we are done. Otherwise we get $v \in W^{1,q_0^*}$, and repeating the argument with $q_0$ replaced by $q_0^*$ we see that $G$ and thus also $v$ are in $W^{1,\min\{q_0^{**},q\}}$. Iterating this procedure finitely many times gives $v \in W^{1,q}$, which proves $u \in W^{1,q}_{\mathrm{loc}}$ upon varying $x_0$. To see \eqref{tboundedinvertible}, note that 
\[
\norm{E_{x_0} w}_{W^{1,p}} \leq C \norm{w}_{W^{-1,p}}, \qquad w \in W^{-1,p}, \ \ p \in [p_0',p_0].
\]
This follows by complex interpolation from the cases $p=p_0$ and $p=p_0'$ given by Lemma \ref{lemma_v_representation_formula}. Consequently 
\begin{align*}
\norm{E_{x_0} Q_{x_0} w}_{W^{1,p}} &\leq C \norm{D_l (\chi_2(\tilde{A}^{lm}-\tilde{A}^{lm}(0)) D_m w)}_{W^{-1,p}} \\
 &\leq C \sum_{l,m} \norm{\tilde{A}^{lm}-\tilde{A}^{lm}(0)}_{L^{\infty}(B(0,2))} \norm{w}_{W^{1,p}}.
\end{align*}
Since $A^{lm}$ are continuous, we can choose $r$ small enough so that 
\[
\norm{E_{x_0} Q_{x_0} w}_{W^{1,p}} \leq \frac{1}{2} \norm{w}_{W^{1,p}}, \qquad p \in [p_0',p_0].
\]
This implies \eqref{tboundedinvertible} by Neumann series. To see \eqref{gwoneq}, note that our assumptions and Sobolev embedding imply 
\[
F \in W^{-1,q} + L^{q_0} + W^{-1,q_0^*}.
\]
Since all the terms are compactly supported, we get 
\[
F \in W^{-1,p}, \qquad p \in (1, \min\{q_0^*,q\}].
\]
The mapping properties of $E_{x_0}$ in Lemma \ref{lemma_v_representation_formula} and the fact that $\psi(D) \chi_{2} v$ is Schwartz imply \eqref{gwoneq}.

Assume next the conditions in (b). We claim that for $r > 0$ small, 
\begin{gather}
\smash{\raisebox{-\dimexpr.5\normalbaselineskip+.5\jot}{$%
            \left\{ \begin{array}{@{}c@{}}\\[\jot]\\[\jot]\end{array}\right.$}}
\text{$T$ is invertible on $W^{2,q}(\mR^n, \mR^N)$, and }\label{tboundedinvertible_a2} \\
\text{$G \in W^{2,q}$.} \label{gwoneq_a2}
\end{gather}
Since also the conditions in (a) hold, we have $v = T^{-1}(G)$ and thus (b) follows immediately from these claims upon varying $x_0$. To prove \eqref{tboundedinvertible_a2}, note that Lemma \ref{lemma_v_representation_formula} gives 
\[
\norm{E_{x_0} w}_{W^{2,q}} \leq C \norm{w}_{L^q}, \qquad w \in L^q.
\]
Combining this with the pointwise multiplier property 
\[
\norm{au}_{W^{1,q}} \leq \norm{au}_{L^q} + \norm{\nabla(au)}_{L^q} \leq C (\norm{a}_{L^{\infty}} + \norm{\nabla a}_{L^t}) \norm{u}_{W^{1,q}}
\]
by using the assumption on $t$ (if $p=n$ this holds for $a$ supported in $B(0,2)$), we get 
\begin{align*}
\norm{E_{x_0} Q_{x_0} w}_{W^{2,q}} &\leq C \sum_{l=1}^n \norm{\chi_2(\tilde{A}^{lm}-\tilde{A}^{lm}(0)) D_m w}_{W^{1,q}} \\
 &\leq C o_{r\to 0}(1) \norm{w}_{W^{2,q}}
\end{align*}
using the fact that $A^{lm} \in C^0_{\mathrm{loc}} \cap W^{1,t}_{\mathrm{loc}}$. This shows \eqref{tboundedinvertible_a2} by Neumann series. To prove \eqref{gwoneq_a2}, consider first the case $1 < q < n$. Then $L^q \subset W^{-1,q^*}$, so $u \in W^{1,q^*}_{\mathrm{loc}}$ by (a). Since $L^n, L^{q^*} \subset L^q$, we get $F \in L^q$ and therefore $G \in W^{2,q}$ by the mapping properties of $E_{x_0}$. In the case $n \leq q < \infty$ we get $L^q_{\mathrm{loc}} \subset W^{-1,p}_{\mathrm{loc}}$ for all $p < \infty$, thus $u \in W^{1,p}_{\mathrm{loc}}$ for all $p < \infty$ by (a), and $u \in L^{\infty}$ by Sobolev embedding. The assumption $A^{lm} \in C^0_{\mathrm{loc}} \cap W^{1,t}_{\mathrm{loc}}$ gives $F \in L^q$ and $G \in W^{2,q}$. This proves \eqref{gwoneq_a2} and also (b).

Assume next the conditions in (c). It is enough to prove that for $r > 0$ small, 
\begin{gather}
\smash{\raisebox{-\dimexpr.5\normalbaselineskip+.5\jot}{$%
            \left\{ \begin{array}{@{}c@{}}\\[\jot]\\[\jot]\end{array}\right.$}}
\text{$T$ is invertible on $C^{1+\alpha}(\mR^n, \mR^N)$, and }\label{tboundedinvertible_c} \\
\text{$G \in C^{1+\alpha}$.} \label{gwoneq_c}
\end{gather}
The main point is that $\abs{A^{lm}(y)-A^{lm}(z)} \leq M \abs{y-z}^{\alpha}$, so we have 
\[
\norm{\chi_2 (\tilde{A}^{lm} - \tilde{A}^{lm}(0))}_{C^{\alpha}} \leq C r^{\alpha}
\]
and consequently 
\begin{align*}
\norm{E_{x_0} Q_{x_0} w}_{C^{1+\alpha}} &\leq C \sum_{l=1}^n\norm{\chi_2(\tilde{A}^{lm}-\tilde{A}^{lm}(0)) D_m w}_{C^{\alpha}} \\
 &\leq C r^{\alpha} \norm{w}_{C^{1+\alpha}}.
\end{align*}
This proves \eqref{tboundedinvertible_c} by Neumann series if $r > 0$ is small enough. To prove \eqref{gwoneq_c} one first notes that $f \in W^{-1,q}_{\mathrm{loc}}$ for all $q$, hence $u \in W^{1,q}_{\mathrm{loc}}$ for all $q$ by (a) and $u \in C^{\beta}_{\mathrm{loc}}$ for $0 < \beta < 1$ by Sobolev embedding. Thus $F \in C^{-1+\alpha}$, so that $G \in C^{1+\alpha}$ by the mapping properties of $E_{x_0}$.

Finally, assume the conditions in (d). We do the proof by induction on $k$. If $k=0$ the result follows from (c). Assume the result holds smoothness indices $\leq k-1$ where $k \geq 1$, and assume that 
\[
A^{lm} \in C^{k+\alpha}_{\mathrm{loc}}, \qquad f \in C^{k-1+\alpha}_{\mathrm{loc}}.
\]
By induction hypothesis, $u \in C^{k+\beta}_{\mathrm{loc}}$ for $\beta < 1$. Also, since $k \geq 1$, part (b) implies that $u \in W^{2,q}_{\mathrm{loc}}$ for any $q < \infty$. Hence it is possible to differentiate the equation once to get 
\[
D_l(A^{lm} D_m(D_j u)) = D_j f - D_l ( (D_j A^{lm}) D_m u ).
\]
The right hand side is in $C^{k-2+\alpha}_{\mathrm{loc}}$, hence the induction hypothesis gives $D_j u \in C^{k+\alpha}_{\mathrm{loc}}$ for all $j$. This proves that $u \in C^{k+1+\alpha}_{\mathrm{loc}}$.
\end{proof}

\bibliographystyle{alpha}

\end{document}